\documentclass[reqno]{amsart}
\usepackage{url}
\usepackage{amssymb,amsthm,amsfonts,amstext,amsmath}
\usepackage{pdfsync}  
\usepackage{newcent}       
\usepackage{helvet}         
\usepackage{courier}        
\usepackage{graphicx}
\usepackage{enumerate}
\usepackage{hyperref}
\numberwithin{equation}{section}
\usepackage{color}
\usepackage[active]{srcltx}
\usepackage{pstricks}
\usepackage{float}
\usepackage{mathrsfs}
\usepackage{tikz}

\newtheorem{theorem}{Theorem}[section]
\newtheorem{lemma}[theorem]{Lemma}
\newtheorem{proposition}[theorem]{Proposition}

\newtheorem{remark}[theorem]{Remark}
\newtheorem{definition}{Definition}
\newcommand{\mc}[1]{{\mathcal #1}}
\newcommand{\mf}[1]{{\mathfrak #1}}

\newcommand{\bb}[1]{{\mathbb #1}}

\newcommand{\eps}{\varepsilon}

\newcommand{\Y}{\mathcal{Y}}
\newcommand{\<}{\langle}
\renewcommand{\>}{\rangle}
\newcommand{\p}{\partial}
\newcommand{\pfrac}[2]{\genfrac{}{}{}{1}{#1}{#2}}

\newcommand{\at}[2]{\genfrac{}{}{0pt}{}{#1}{#2}}

 \newcommand{\A}{C^{2}[0,1]}

\newcommand{\bola}{\textrm{\Large\textperiodcentered}}

 \newcommand{\SN}{\mc S_{\textrm{\rm Neu}}(\bb R)}
\newcommand{\Sc}{\mc S(\bb R)}

\let\oldtocsection=\tocsection
\let\oldtocsubsection=\tocsubsection
\let\oldtocsubsubsection=\tocsubsubsection
\renewcommand{\tocsection}[2]{\hspace{0em}\oldtocsection{#1}{#2}}
\renewcommand{\tocsubsection}[2]{\hspace{1em}\oldtocsubsection{#1}{#2}}
\renewcommand{\tocsubsubsection}[2]{\hspace{2em}\oldtocsubsubsection{#1}{#2}}
\DeclareRobustCommand{\SkipTocEntry}[5]{}


\begin{document}

\title[SSEP with a slow site]{Scaling limits for the exclusion \\ process with a slow site}

\author{Tertuliano Franco}
\address{UFBA\\
 Instituto de Matem\'atica, Campus de Ondina, Av. Adhemar de Barros, S/N. CEP 40170-110\\
Salvador, Brazil}
\curraddr{}
\email{tertu@ufba.br}
\thanks{}

\author{Patr\'{\i}cia Gon\c{c}alves}
\address{\noindent Departamento de Matem\'atica, PUC-RIO, Rua Marqu\^es de S\~ao Vicente, no. 225, 22453-900, Rio de Janeiro, Rj-Brazil and CMAT, Centro de Matem\'atica da Universidade do Minho, Campus de Gualtar, 4710-057 Braga, Portugal.}
\email{patricia@mat.puc-rio.br}

\author{Gunter M. Sch\"utz}
\address{Institute of Complex Systems II, Forschungszentrum J\"ulich, 52428 J\"ulich, Germany}
\curraddr{}
\email{g.schuetz@fz-juelich.de}
\thanks{}


\begin{abstract}
We consider the symmetric simple exclusion processes with a slow site in the discrete torus with $n$ sites.
In this model, particles perform nearest-neighbor symmetric random walks with jump rates everywhere equal
to one, except at one particular site, {\em{the slow site}}, where the jump rate of entering that site is equal to one, but the jump rate of leaving that site is given by a parameter $g(n)$. Two cases are treated, namely  $g(n)=1+o(1)$, and $g(n)=\alpha n^{-\beta}$ with $\beta>1$, $\alpha>0$.
In the former, both the  hydrodynamic behavior
and equilibrium fluctuations are driven by the heat equation (with periodic boundary conditions when in finite volume). In the latter, they are driven by the heat equation with  Neumann boundary conditions. We therefore establish the existence of a dynamical phase transition. The critical behavior remains  open.
\end{abstract}

\maketitle

\tableofcontents

\section{Introduction}
In the seventies, Dobrushin and Spitzer, see \cite{Spi} and references therein, initiated the idea of obtaining a mathematically precise understanding
of the emergence of macroscopic behavior in gases or fluids from
the microscopic interaction of a large number of identical particles with stochastic dynamics. This approach has
turned out to be extremely fruitful both in probability theory and statistical physics (e.g. see  the books  \cite{Spohn,kl})  and it still raises attention nowadays.
In this context, recent studies have been made in hydrodynamic limit/fluctuations of interacting particle systems in
random/non homogeneous medium, see for instance \cite{fl,fgn1,fgn2,jlt} and references therein.

So far, most of the work done in this field concerns the bulk hydrodynamics, i.e.,
the derivation of macroscopic partial differential equations
arising from the \emph{bulk interactions} of the underlying particle system. To this end,
one usually considers an infinite system or  a finite torus with periodic
boundary conditions and then takes the thermodynamic limit.
However, in applications to physical systems one is usually confronted with finite
systems, which requires the study of a partial differential equation on
a finite interval with prescribed boundary conditions. This raises the question from which microscopic
\emph{boundary interactions} a given type of boundary condition emerges at the macroscopic scale.

This is an important issue both for  boundary-driven open systems, 
where boundary
interactions can induce long-range correlations \cite{Spoh83}
and bulk phase transitions due to the absence of particle conservation
at the boundaries
\cite{Brza07}, and for bulk-driven
conservative systems on the torus where even a single defect bond between two neighboring sites
can change bulk relaxation behavior or
lead to macroscopic discontinuities in the hydrostatic density profiles.\footnote{See \cite{Jano92,Schu93,Baha04,Chen08,Basu14,Popk15} for numerical, exact and rigorous results
for the asymmetric simple exclusion process and \cite{Schu01,Scha10} for a review, including experimental applications of interacting particle systems with boundary interactions in physical and biological
systems.} Given such rich behaviour
due to boundary effects in non-conservative or bulk-driven systems
it is natural to explore the macroscopic role of a microscopic defect on
a torus in a conservative system in the \emph{absence} of bulk-driving and to ask whether such
a defect can be described on macroscopic scale in terms of a boundary condition
for the PDE describing the bulk hydrodynamics.

%
In this work we address this problem for the symmetric simple exclusion process (SSEP) 
on the discrete torus in the presence of a defect site.
The model can be described as follows. Each site of the discrete torus with $n$ sites, that we denote by
$\bb T_n= \bb Z/n\bb Z$,  is allowed to have at most one particle.
To each site is associated a Poisson
clock, all of them being independent.  If there is a particle in the associated site, this
particle chooses one of its nearest neighbors with equal probability when the clock rings. 
If the chosen site is empty, the particle
jumps to it. Otherwise nothing happens. All sites have a Poisson clock of parameter two, except the origin,
which has a Poisson clock of parameter $2g(n)$. If $g(n)<1$, the origin behaves as a
\emph{trap}, and (in average) it keeps a particle there for a longer time than the other sites do. We call this site a \emph{slow site}. The main results of the present work are the  hydrodynamic limit and the  equilibrium fluctuations for the exclusion process with such a
slow site. 

\begin{figure}[H]
  \centering
  \includegraphics{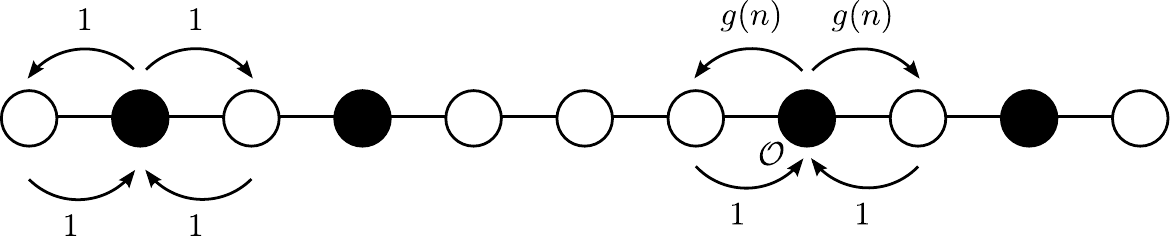}
  \label{fig:1}
\caption{Exclusion process with a {\emph{slow site}}.}
\end{figure}

 Specifically, for $g(n)=1+o(1)$ it is shown here that the limit for the time trajectory of the spatial density of particles is given by the solution of the heat equation with periodic boundary conditions, namely:
 \begin{equation}\label{PDE1}
\begin{cases}
 \partial_t \rho(t,u) \; =\; \p_{u}^2 \rho(t,u)\,,&t \geq 0,\, u\in \bb T\,,\\
  \rho(0,u) \;=\; \rho_0(u)\,, &u \in \bb T\,,\\
\end{cases}
\end{equation}
where $\bb T$ is the one-dimensional continuous torus. 
 
 Moreover, considering the same particle system evolving on $\bb Z$, we prove that the equilibrium fluctuations of the system are driven by a generalized Ornstein-Uhlenbeck process $\mc Y_t$ which is the solution of
  \begin{equation*}
d\mathcal{Y}_t= \Delta \mathcal{Y}_tdt+\sqrt{2\chi(p)} \nabla d\mc{W}_t\,
\end{equation*}
where $\mc{W}_t$ is a Brownian motion on the  space $\mc S'(\bb R)$ of tempered distributions and $\chi(p)$ is the compressibility, which is a coefficient related to the invariant measure of the system.  Both results are true irrespective of whether
$o(1)$ is positive or negative, i.e., of whether the origin is a fast site or a slow site.

 On the other hand,
if $g(n)=\alpha n^{-\beta}$, $\alpha>0$, $\beta>1$, the limit for the time trajectory of the spatial density of particles is given by the solution of the heat equation with Neumann boundary conditions, namely:
 \begin{equation}\label{PDE3}
\begin{cases}
 \partial_t \rho(t,u) \; =\; \p_{u}^2 \rho(t,u)\,,&t \geq 0,\, u\in (0,1)\,,\\
 \partial_u \rho(t,0^+) \; =\;\partial_u \rho(t,0^-)= 0\,,&t \geq 0\,,\\
  \rho(0,u) \;=\; \rho_0(u)\,, &u \in (0,1)\,,\\
\end{cases}
\end{equation}
where $0^+$ and $0^-$ denotes right and left side limits, respectively. This represents no passage of particles in the continuum limit.

 Again considering the same particle system evolving on $\bb Z$, we prove that the equilibrium fluctuations of the system when $g(n)=\alpha n^{-\beta}$, $\alpha>0$,  $\beta>1$ are driven by the solution of
  \begin{equation*}
d\mathcal{Y}_t= \Delta_{\textrm{Neu}} \mathcal{Y}_tdt+\sqrt{2\chi(p)} \nabla_{\textrm{Neu}} d\mc{W}_t\,,
\end{equation*}
which is essentially a version of the previous generalized Ornstein-Uhlenbeck process associated to the PDE \eqref{PDE3}, in the same setting of \cite{fgn2}. These Ornstein-Uhlenbeck processes are precisely stated in Section \ref{s2}.  

 We point out that a similar model with conservative dynamics on the torus has 
been considered in \cite{fgn1,fgn2,fl}, which consists in the SSEP with a \emph{slow bond} of intensity
$g(n)=\alpha n^{-\beta}$, $\alpha> 0$, $\beta\geq 0$. In that model particles perform nearest-neighbor symmetric random walks, whose jump rate is equal to one at all bonds, except at a particular \emph{bond}, where it is equal to $g(n)$. From \cite{fgn1,fgn2,fl} it is known that  hydrodynamic limit/fluctuations for exclusion processes with a slow bond have three
different behaviors depending on the regime of $\beta$.


 For the SSEP with a slow site that we treat here the methods used for the slow bond
problem cannot be adapted in any straightforward fashion as the asymmetry at the slow site gives rise to novel difficulties in the study of its hydrodynamic behavior and fluctuations. The model is reversible, as is the SSEP with a slow bond, but it is not \emph{self-dual, in contrast} to the SSEP with a slow bond. Moreover, the invariant measures for the SSEP with a slow site are not translation invariant, as happens for the SSEP with a slow bond. As a consequence, the proof of the hydrodynamic limit for the SSEP with a slow site requires   different approaches  from the ones of \cite{fgn1,fgn2,fl} and one cannot naively
extend the results obtained for the slow bond to the case of the slow site.

As described above, in this paper we are able to  characterize the hydrodynamic limit and the equilibrium fluctuations for $g(n)=\alpha n^{-\beta}$ when $\alpha>0$, $\beta>1$. The case $0\leq \beta\leq 1$ remains open. However, we present and motivate a conjecture on the behavior of the system in that case. Moreover, since we present also the hydrodynamic limit and the equilibrium fluctuations for $g(n)$ close to one, namely  $g(n)=1+o(1)$, the existence of a dynamical phase transition in the behaviour of the system from periodic boundary conditions to  Neumann boundary conditions at a critical value of $\beta$ in the range $0\leq \beta\leq 1$
is  established.

In order to put our results into a broader perspective conservative
particle systems with defects we point out that for a single slow bond the 
SSEP treated \cite{fgn1,fgn2,fl} exhibits the same hydrodynamic behaviour
as non-interacting random walks with a single slow bond. On the other hand, with a 
single slow site
both hydrodynamic limit and fluctuations of non-interacting particles would be driven by a disconnect behavior for any $\beta>0$,
leading to Dirichlet boundary conditions with boundary densities 0. Thus the similarity between the SSEP and non-interacting
particles that one finds for a slow bond breaks down for a slow site, adding further motivation for a detailed investigation
of  the SSEP with a slow site.



It is also worthwhile to compare our result with a result derived in \cite{jlt} for a related problem. The model  considered there is called the \emph{Bouchaud} trap model. In that model, particles perform independent random walks in a  random environment with traps given by i.i.d. alpha-stable random variables. In \cite{jlt} it was proved that the hydrodynamic limit for such  model is given by a generalized partial differential equation depending on an alpha-stable subordinator. The present paper suggests that a trap model of exclusion type should not have the same limit as obtained in \cite{jlt} for a trap model of independent random walks. 
 For the SSEP the asymmetry at the slow site yields a limit that has some properties in common with a slow bond and therefore a  behavior completely different from the one observed in \cite{jlt}.

Here follows the outline of this paper. In Section \ref{s2} we give notations, precise definitions and statements of the results. In Section \ref{s3} we present the hydrodynamic limit of the model. In Section \ref{s4} we present the equilibrium fluctuations (in infinite volume). In Section \ref{s5} we state a conjecture on what should be the complete scenario for exclusion processes with a slow site. In Section \ref{s6} we  present an extra result on the hydrodynamic behavior of the SSEP with $k$ neighboring slow bonds, which we use as an argument to sustain our conjecture in Section \ref{s5}.

\section{Statement of results}\label{s2}

\subsection{The model}
 A particle system can be constructed through its generator or via Poisson processes. In this work we will make use of both.

Let  $\bb T_n=\bb Z /n\bb Z=\{0,1,\ldots,n-1\}$ be the one-dimensional discrete torus with $n$ points. The simple symmetric exclusion process (SSEP) with a  slow site is
 the Markov process with state space $\{0,1\}^{\bb T_n}$ and with generator $\mf L_{n}$ acting on functions $f:\{0,1\}^{\bb T_n}\rightarrow \bb{R}$ as
\begin{equation}\label{generator}
\mf L_{n}f(\eta)=\sum_{\substack{x,y\in \bb T_n\\ |x-y|\leq 1}}\,\xi^{n}_x\,\eta(x)(1-\eta(y))\,[f(\eta^{x,y})-f(\eta)]\,,
\end{equation}
where the jump rates $\xi^n_x$ are given by
\begin{equation*}
\xi^n_x=\begin{cases}
g(n)\,,&\textrm{ if } \quad x=0\,,\\
1\,,&\textrm{ if }\quad  x\in \bb T_n\backslash\{0\} \,,\\
\end{cases}
\end{equation*}
where $g(n)>0$ and $\eta^{x,x+1}$ is the configuration obtained from $\eta$ by exchanging the occupation variables $\eta(x)$ and $\eta(x+1)$. Formally,
\begin{equation}\label{eta}
(\eta^{x,x+1})(y)=
\begin{cases}
\eta(x+1),& \mbox{ if }\quad y=x\,,\\
\eta(x),& \mbox{ if } \quad y=x+1\,,\\
\eta(y),& \mbox{ otherwise.}
\end{cases}
\end{equation}

Its dynamics can be described as follows. To each site we attach two Poisson processes, one corresponding to jumps from $x$ to $x+1$ and the other corresponding to jumps from $x$ to $x-1$. If the site $x$ is occupied and the site  $x+1$ is empty, the particle moves 
from site $x$ to site $x+1$ at a time arrival of the Poisson process associated to $\{x,x+1\}$,
and analogously for sites $\{x,x-1\}$. The jump rates corresponding to those transitions are shown in Figure \ref{fig:1}. 

For fixed $n$, let $\{\eta_\tau : \tau\ge 0\}$ be
the  Markov process with generator $\mf L_n$. Notice that  $\eta_\tau$ depends on $g(n)$, 
but we do not display this dependence in the notation. 
We denote by $\{\eta_t : t\ge 0\}$ 
the  Markov process with generator $n^2\mf L_n$.
This time factor $n^2$ is the so-called \emph{diffusive time scaling}. We observe that this is equivalent to define $\eta_t:=\eta_{n^2\tau}$.

Next we establish a family of invariant measures (in fact, reversible) for the dynamics introduced above.
\begin{proposition}\label{prop:2.1}
For any $p\in[0,1]$, the Bernoulli product measure $\nu_p$ on the space $\{0,1\}^{\bb T_n}$ with marginals given by
\begin{equation}\label{nub}
\nu_p\{\eta\,;\,\eta(x)=1\}\;=\;m_p(x)\;=\;
\begin{cases}
\displaystyle\frac{\pfrac{p}{g(n)}}{(1-p)+\pfrac{p}{g(n)}}\;,&\textrm{ if }\quad x=0\,,\\
p\,,&\textrm{ if }\quad x\in \bb T_n\backslash\{0\}\,,\\
\end{cases}
\end{equation}
is reversible for the Markov process $\{\eta_\tau : \tau\ge 0\}$.
\end{proposition}
The proof of this proposition consists only in checking the detailed balance equation, which is straightforward and for that reason it will be omitted.

Above and in what follows, a sub-index in a function means a variable, \emph{not a derivative}.
Denote by $\bb T$ the one-dimensional continuous torus $\bb R/\bb Z=[0,1)$ and by $\<\cdot,\cdot\>$ the inner product in $L^2[0,1]$.

\subsection{Hydrodynamics}
\begin{definition}\label{heat equation Per} Let $\rho_0:\bb T\to [0,1]$ be a measurable function.
We say that $\rho$ is a weak solution of the heat equation with periodic boundary conditions given by
\begin{equation}\label{hep}
\left\{
\begin{array}{ll}
 \partial_t \rho(t,u) \; =\; \partial^2_u \rho(t,u)\,, &t \geq 0,\, u\in \bb T\,,\\
  \rho(0,u) \;=\; \rho_0(u), &u \in \bb T\,,
\end{array}
\right.
\end{equation}
if, for  all $t\in [0,T]$ and for all  $H\in C^2(\bb T)$,
\begin{equation}\label{eqint2}
\begin{split}
\< \rho_t,H_t\>\!-\!\<\rho_0,H_0\>
\!-\!\! \int_0^t\!\!\!\big\< \rho_s, \,& \partial^2_u H_s\big\> ds\;=\;0\,.
\end{split}
\end{equation}
\end{definition}
Next, we define what we mean by weak solutions of the heat equation with  Neumann boundary conditions, as given in \eqref{PDE3}.
We introduce first some technical background on Sobolev spaces.

\begin{definition}\label{Sobolevdefinition}
Let  $\mc H^1$ be the set of all  $L^1$ functions $\zeta: [0,1]\to\bb R$ such that
there exists a function $\p_u\zeta\in L^2[0,1]$ satisfying
\begin{equation*}
 \<\partial_uG,\zeta\>\,=\,-\<G,\partial_u\zeta\>\,,
\end{equation*}
for all $G\in C^{\infty}[0,1]$ with compact support contained in $(0,1)$.
For $\zeta\in\mc H^1$, we define the norm
\begin{equation*}
 \Vert \zeta\Vert_{\mc H^1}\,:=\, \Big(\Vert \zeta\Vert_{L^2[0,1]}^2+\Vert\partial_u\zeta\Vert_{L^2[0,1]}^2\Big)^{1/2}\,.
\end{equation*}
Let $L^2(0,T;\mc H^1)$ be the space of all measurable functions $\xi:[0,T]\to \mc H^1$ such that
\begin{equation*}
\Vert\xi \Vert_{L^2(0,T;\mc H^1)}^2 \,
:=\,\int_0^T \Vert \xi_t\Vert_{\mc H^1}^2\,dt\,<\,\infty\,.
\end{equation*}
\end{definition}


Abusing notation slightly, we denote by $\A$  the set of functions $H:\bb T\to\bb R$ that are continuously twice differentiable in $\bb T\backslash\{0\}$ and have a $C^2$-extension  to the closed interval $[0,1]$.
\begin{definition}\label{heat equation Neumann}
Let $\rho_0:\bb T\to [0,1]$ be a measurable function.
 We say that $\rho$ is a weak solution of the heat equation with Neumann boundary conditions
 \begin{equation}\label{hen}
\left\{
\begin{array}{ll}
 \partial_t \rho(t,u) \; =\; \p^2_u \rho(t,u)\,, &t \geq 0,\, u\in (0,1)\,,\\
 \partial_u \rho(t,0^+) \; =\;\partial_u \rho(t,0^-)= 0\,, \qquad &t \geq 0\,,\\
 \rho(0,u) \;=\; \rho_0(u), &u \in (0,1)\,,
\end{array}
\right.
\end{equation}
if $\rho$ belongs to $L^2(0,T;\mathcal{H}^1)$ and for all $t\in [0,T]$ and for all  $H\in \A$,
\begin{equation*}
\begin{split}
\< \rho_t\,,\,H\>-\<\rho_0\,,\,H\> & - \int_0^t\big\< \rho_s\,,\,\p^2_u H\big\>\, ds\\
& -\int_0^t\big(\rho_s(0^+)\,\partial_u H(0^+)-\rho_s(0^-)\,\partial_u H(0^-)\big)\,ds\;=\;0\,.\\
\end{split}
\end{equation*}
\end{definition}

Let $\mc D(\bb R_+, \{0,1\}^{\bb T_n})$ be the path space of
c\`adl\`ag\footnote{From the French, ``continuous from  the right with limits from the left''.} trajectories with values in $\{0,1\}^{\bb T_n}$. For a
measure $\mu_n$ on $\{0,1\}^{\bb T_n}$, denote by $\bb P_{\mu_n}$ the
probability measure on $\mc D(\bb R_+, \{0,1\}^{\bb T_n})$ induced by the
initial state $\mu_n$ and the Markov process $\{\eta_t : t\ge 0\}$. Notice that in fact $\bb P_{\mu_n}=\bb P_{\mu_n}^{g(n),n}$ but  we will not carry the dependence on $n$ nor $g$ in order to not overload notation. By  $\bb E_{\mu_n}$ we mean the expectation with respect to $\bb P_{\mu_n}$.

The notation $\eta_\bola$ is reserved to represent elements of the Skorohod space $\mc D(\bb R_+, \{0,1\}^{\bb T_n})$, i.e., time trajectories of the exclusion process with a slow site. This notation $\eta_\bola$ should not be confused with the notation $\eta$ for elements of $\{0,1\}^{\bb T_n}$.

From now on we fix a  profile $\gamma : \bb T \to {[0,1]}$, representing the initial density of particles. To avoid uninteresting technical complications, we assume that $\gamma$ is continuous at all $x\in \bb T\backslash \{0\}$ and bounded from below by a positive constant: 
\begin{equation}\label{assumption}
\zeta\;:=\;\inf_{x\in\bb T}\gamma(x)\;>\;0.
\end{equation}

\begin{theorem}\label{th:1}
For each $n\in\bb N$, let $\mu_n$ be a Bernoulli product measure on $\{0,1\}^ {\bb T_n}$ with marginal distributions given by
\begin{equation}\label{eq288}
\mu_n\{\eta\,;\,\eta(x)=1\}=\gamma(\pfrac{x}{n})\,.
\end{equation}
   Then, for any $t>0$, for every $\delta>0$ and every $H\in C(\bb{T})$, it holds that
\begin{equation}\label{eq:2.2}
\lim_{n\to\infty}
\bb P_{\mu_n} \Big\{\eta_\bola : \, \Big\vert \pfrac{1}{n} \sum_{x\in\bb{T}_n}
H(\pfrac{x}{n})\, \eta_{t}(x) - \int_{\bb T} H(u)\, \rho(t,u) du \Big\vert
> \delta \Big\} \;=\; 0\,,
\end{equation}
 where
\vspace{0.2cm}
 \begin{itemize}
\item
for $g(n)=1+o(1)$,  $\rho$ is the unique weak solution of \eqref{hep};
\vspace{0.2cm}

\item  for $g(n)=\alpha n^{-\beta}$, $\alpha>0$, $\beta>1$,  $\rho$ is the unique weak solution of  \eqref{hen};
\vspace{0.2cm}
\end{itemize}
and where, in both cases, the initial condition of the corresponding partial differential equation is given by $\rho_0=\gamma$.
\vspace{0.2cm}
\end{theorem}

\begin{remark}\rm About the constants: To avoid repetitions along the paper, we fix, once and for all, the assumptions $\alpha>0$ and $\beta>1$.
\end{remark}

\begin{remark}\rm
  About the statement of the Theorem: If at the initial time the density of particles converges to the profile $\gamma(\cdot)$, then, in the future time $t$, the density of particles converges to a profile $\rho(t,\cdot)$ which is the weak solution of the heat equation with the corresponding boundary conditions  and with initial condition $\rho_0=\gamma$.
\end{remark}
\begin{remark}\rm
About the scaling: In the claim  of the Theorem \ref{th:1} one can see that the space is rescaled by $n^{-1}$ (space between sites) and  time is rescaled by $n^2$, since the ``future time'' is indeed $tn^2$. This is the \emph{diffusive} time scaling.
\end{remark}
\begin{remark}\rm
About the initial measure: We can weaken the hypothesis on  $\mu_n$ by dropping the condition of being a product measure,  and assuming that $\{\mu_n\}_{n\in\bb N}$ is \emph{associated to} $\gamma(\cdot)$, see \cite{kl}. In that case the statement of Theorem \ref{th:1} remains in force. However, this hypothesis would complicate  the attractiveness tools at Subsection \ref{LGN} and for this reason we assume \eqref{eq288}.
\end{remark}
\begin{remark}\rm
About the weak solution: The weak solution of  \eqref{hen} is a function defined on the interval $[0,1]$, not on the torus. But, as already explained, Lebesgue almost sure, it is the same. Thus it makes sense to integrate $\rho$ in the torus $\bb T$, as it appears in the equation
\eqref{eq:2.2}.
\end{remark}

\subsection{Equilibrium density fluctuations}
In this section we consider $\eta_t$ evolving on the one-dimensional lattice $\bb Z$ and starting from the invariant state $\nu_p$, with $p\in(0,1)$. Therefore, the generator of the process is given  by \eqref{generator} with $\bb T_n$ replaced by $\bb Z$, namely
\begin{equation*}
\mf L_{n}f(\eta)=\sum_{\substack{x,y\in \bb Z\\ |x-y|\leq 1}}\,\xi^{n}_x\,\eta(x)(1-\eta(y))\,[f(\eta^{x,y})-f(\eta)]\,,
\end{equation*}
for local functions $f:\{0,1\}^{\bb Z}\rightarrow \bb{R}$.

 From now on we fix $p\in(0,1)$.
In order to establish the central limit theorem (C.L.T.) for the density under the invariant state $\nu_{p}$,
we need to introduce the  fluctuation field as the linear functional acting on test functions $H$ as
\begin{equation}\label{flut}
 \mc Y^n_{t}(H)= \frac{1}{\sqrt n}\sum_{x\in{\mathbb{Z}}}H\Big(\frac{x}{n}\Big)(\eta_{t}(x)-m_p(x)), 
\end{equation}
where $m_p(x)$ is the mean of $\eta_t(x)$ with respect to $\nu_p$ introduced in \eqref{nub}. We emphasize that  $\{\eta_t : t\ge 0\}$ is
the  Markov process with generator $n^2\mf L_n$. Denote by $\bb P_p$ the probability measure on  the Skorohod path space $\mc D(\bb R_+, \{0,1\}^{\bb Z})$ induced by the initial state $\nu_p$ and the Markov process $\{\eta_t : t\ge 0\}$ and we denote by $\bb E_p$ the expectation  with respect to $\bb P_p$.

Now we introduce the space of test functions. Since the hydrodynamics is governed by different partial differential equations, the state space for $H$  depends on the jump rate $g(n)$ that we defined at the slow site.

\begin{definition}\label{schwartz} 
Let $\mc S(\bb R)$ be the usual Schwartz space of functions $H:\bb R\to\bb R$ such that $H\in C^\infty(\bb R)$  and 
\begin{equation*}
 \Vert H \Vert_{k,\ell}\;:=\;\sup_{x\in \bb R}|(1+|x|^\ell)
\,\frac{d^kH}{dx^k}(x)|\;<\;\infty\,,
\end{equation*}
for all integers $k,\ell\geq 0$. We define  $\mc S_{\textrm{Neu}}(\bb R)$ as the space composed  of functions $H:\bb R\to\bb R$ such that
\begin{enumerate}
\item Except possibly  at $x=0$, the function $H$ is continuous and infinitely differentiable,
\item The function $H$ is continuous from the right at zero,
\item  For all integers $k,\ell\geq 0$,
\begin{equation*}
 \Vert H \Vert_{k,\ell,+}\;:=\;\sup_{x> 0}|(1+|x|^\ell)
\,\frac{d^k}{dx^k}H(x)|\;<\;\infty\,,
\end{equation*}
and 
\begin{equation*}
 \Vert H \Vert_{k,\ell,-}\;:=\;\sup_{x< 0}|(1+|x|^\ell)
\,\frac{d^kH}{dx^k}(x)|\;<\;\infty\,,
\end{equation*}
\item For any integer $k\geq 0$, 
\[
\lim_{x\to 0^+ }\frac{d^{2k+1} H}{dx^{2k+1}}(x)\;=\; 
\lim_{x\to 0^- }\frac{d^{2k+1} H}{dx^{2k+1}}(x)
\;=\;0.
\]
\end{enumerate}
\end{definition}
Notice that it is not required that $H$ is continuous at $x=0$.  Intuitively, this space $\SN$ corresponds to two independent Schwartz spaces in each half line. The chosen notation comes from the expression  \textit{Neumann boundary conditions}. 

Both spaces $\Sc$ and $\SN$ are Fr\'echet spaces. The proof that $\Sc$ is Fr\'echet can be found in  \cite{rs}, for instance. The proof that $\SN$ is Fr\'echet is quite similar and will be omitted. 

The set of continuous linear functions $f:\Sc\to\bb R$ and $f:\SN\to\bb R$ with respect to the topology generated by the corresponding semi-norms will be denoted by $\mc S'(\bb R)$ and  
$\mc S_{\textrm{Neu}}'(\bb R)$, respectively.

The notation $\nabla$ and $\Delta$ mean the first and second space derivatives. In the case of $\SN$, we will make use of the following definition:
\begin{definition}
We define the operators  $\nabla_{\textrm{Neu}}: \SN \rightarrow \SN$  and
$\Delta_{\textrm{Neu}}: \SN\rightarrow \SN$ by   
\begin{equation*}
\nabla_{\textrm{\rm Neu}} H(u)\;=\;\begin{cases}
\frac{dH}{du}(u), &  \mbox{if}\,\,\,\,u\neq 0\,,\\
\lim_{u\to 0^+ }\frac{dH}{du}(u), &\mbox{if}\,\,\,\,u=0\,,
\end{cases}
\end{equation*}
\begin{equation*}
\Delta_{\textrm{\rm Neu}} H(u)\;=\;\begin{cases}
\frac{d^2H}{du^2}(u), &  \mbox{if}\,\,\,\,u\neq 0\,,\\
\lim_{u\to 0^+ }\frac{d^2H}{du^2}(u), &\mbox{if}\,\,\,\,u=0\,,
\end{cases}
\end{equation*}  
\end{definition}  
\noindent Notice that these operators are  essentially the first and second space derivatives, but defined  in specific domains, which changes the meaning of the operator. Roughly speaking, the operator $\Delta_{\textrm{\rm Neu}}$ is the operator associated to a system blocked at the origin. 
\subsection{Ornstein-Uhlenbeck process}

Denote by $\chi(p)=p(1-p)$ the so-called \textit{static compressibility} of the system.
Based on \cite{HS, kl}, we have a characterization of the generalized Ornstein-Uhlenbeck process, which is a  solution of
\begin{equation}\label{eq Ou}
d\mathcal{Y}_t= \Delta \mathcal{Y}_tdt+\sqrt{2\chi(p)} \nabla d\mc{W}_t\,,
\end{equation}
 where $d\mc{W}_t$ is a space-time white noise of unit variance, in terms of a martingale problem. We will see later that this process, which take values on $\mc S'(\bb R)$, governs the equilibrium fluctuations of the density of particles when the strength of the slow site is given by $g(n)=1+o(1)$. 
 
 On the other hand, when the strength is given by $g(n)=\alpha n^{-\beta}$, the corresponding  
Ornstein-Uhlenbeck process will be the solution of 
\begin{equation}\label{eq Ou Neu}
d\mathcal{Y}_t= \Delta_{\textrm{\rm Neu}} \mathcal{Y}_tdt+\sqrt{2\chi(p)} \nabla_{\textrm{\rm Neu}} d\mc{W}_t\,,
\end{equation}
and taking values on $\mc S'_{\textrm{\rm Neu}}(\bb R)$.

In what follows  $\mathcal{D}([0,T],\mathcal{S}'(\mathbb{R}))$
(resp. $\mathcal{C}([0,T],\mathcal{S}'(\mathbb{R}))$)
 is the space of  c\`adl\`ag (resp. continuous) $\mathcal{S}'(\mathbb{R})$ valued functions endowed with the Skohorod topology. Analogous definitions hold for   $\mathcal{D}([0,T],\mathcal{S}'_{\textrm{\rm Neu}}(\mathbb{R}))$ and  $\mathcal{C}([0,T],\mathcal{S}'_{\textrm{\rm Neu}}(\mathbb{R}))$. 
 
 The rigorous meaning of equations \eqref{eq Ou} and \eqref{eq Ou Neu} is given in terms of the two next propositions.   
 Denote by $T_t:\Sc\to \Sc$ the semi-group of the heat equation in the line (see \cite{fgn3} for instance). It is well known that

\begin{proposition}\label{pp1}
There exists an unique random element $\mc Y_\bola$ taking values in the space $\mc C([0,T],\mathcal{S}'(\bb R))$ such
that:
\begin{itemize}
\item[i)] For every function $H \in \mathcal{S}(\bb R)$, $\mc M_t(H)$ and $\mc N_t(H)$, given by
\begin{equation}\label{lf1}
\begin{split}
&\mc M_t(H)= \mc Y_t(H) -\mc Y_0(H) -  \int_0^t \mc Y_s(\Delta H)ds\,,\\
&\mc N_t(H)=\big(\mc M_t(H)\big)^2 - 2\chi(p) \; t\,\|\nabla H\|_{L^2(\bb R)}^2,
\end{split}
\end{equation}
are $\mc F_t$-martingales, where for each $t\in{[0,T]}$, $\mc F_t:=\sigma(\mc Y_s(H); s\leq t,  H \in \mathcal{S}(\bb R))$.

\item[ii)] $\mc Y_0$ is a Gaussian field of mean zero and covariance given on $G,H\in{\mathcal{S}(\mathbb{R})}$ by
\begin{equation}\label{eq:covar1}
\mathbb{E}_p\big[ \mc Y_0(G) \mc Y_0(H)\big] =  \chi(p)\int_{\mathbb{R}} G(u) H(u) du\,.
\end{equation}
\end{itemize}
Moreover, for each $H\in\mc S(\bb R)$, the stochastic process $\{\Y_t(H)\,:\,t\geq 0\}$ is  Gaussian, being the
distribution of $\Y_t(H)$  conditionally to
$\mc F_s$, for $s<t$, normal  of mean $\Y_s(T_{t-s}H)$ and variance $\int_0^{t-s}\Vert \nabla T_{r}
H\|_{L^2(\bb R)}^2\,dr$.
\end{proposition}

 We call the  random element $\mc Y_\cdot$ the generalized Ornstein-Uhlenbeck process of  
characteristics $\nabla$ and $\Delta$.  From the second equation in \eqref{lf1} and 
L\'evy's Theorem on the martingale characterization of Brownian motion, the process
\begin{equation}\label{Bmotion}
\big(2\chi(p)\|\nabla H\|_{L^2(\bb R)}^2\big)^{-1/2} \mc M_t(H)
\end{equation}
is a standard Brownian motion. Therefore, in view of  Proposition \ref{pp1}, it makes sense to say that $\Y_\bola$ is
the formal solution of \eqref{eq Ou}.

Now, let $T_t^{\textrm{\rm Neu}}:\SN\to\SN$ be the semi-group associated to the following partial differential equation with Neumann boundary conditions:
\begin{equation}\label{pde2}
\left\{
\begin{array}{ll}
 \partial_t u(t,x) = \; \partial^2_{x} u(t,x), &t \geq 0,\, x \in \mathbb R\backslash\{0\}\\
\p_x u(t,0^+)=\p_x u(t,0^-)=0  &t \geq 0\\
 u(0,x) = \; H(x), &x \in \mathbb R.
\end{array}
\right.
\end{equation}
 See \cite{fgn3} for an explicit expression of $T_t^{\textrm{\rm Neu}}$. In a similar way, we have
\begin{proposition}[See \cite{fgn3}]\label{pp2}
There exists an unique random element $\mc Y_\bola$ taking values in the space $\mc C([0,T],\mathcal{S}'_{\textrm{\rm Neu}}(\bb R))$ such
that:
\begin{itemize}
\item[i)] For every function $H \in \mathcal{S}_{\textrm{\rm Neu}}(\bb R)$, $\mc M_t(H)$ and $\mc N_t(H)$ given by
\begin{equation}\label{lf11}
\begin{split}
&\mc M_t(H)= \mc Y_t(H) -\mc Y_0(H) -  \int_0^t \mc Y_s(\Delta_{\textrm{\rm Neu}} H)ds\,,\\
&\mc N_t(H)=\big(\mc M_t(H)\big)^2 - 2\chi(p) \; t\,\|\nabla_{\textrm{\rm Neu}} H\|_{L^2(\bb R)}^2
\end{split}
\end{equation}
are $\mc F_t$-martingales, where for each $t\in{[0,T]}$, $\mc F_t:=\sigma(\mc Y_s(H); s\leq t,  H \in \mathcal{S}_{\textrm{\rm Neu}}(\bb R))$.

\item[ii)] $\mc Y_0$ is a Gaussian field of mean zero and covariance given on $G,H\in{\mathcal{S}_{\textrm{\rm Neu}}(\mathbb{R})}$ by
\eqref{eq:covar1}.
\end{itemize}
Moreover, for each $H\in\mc S_{\textrm{\rm Neu}}(\bb R)$, the stochastic process $\{\Y_t(H)\,;\,t\geq 0\}$ is  Gaussian, being the
distribution of $\Y_t(H)$  conditionally to
$\mc F_s$, for $s<t$, normal  of mean $\Y_s(T_{t-s}^{\textrm{\rm Neu}}H)$ and variance $\int_0^{t-s}\Vert \nabla T_{r}^{\textrm{\rm Neu}}
H\|_{L^2(\bb R)}^2\,dr$.
\end{proposition}
 We call the  random element $\mc Y_\cdot$ the generalized Ornstein-Uhlenbeck process of  
characteristics $\nabla_{\textrm{\rm Neu}}$ and $\Delta_{\textrm{\rm Neu}}$.
We are in position to state our result for the fluctuations of the density of particles.
\begin{theorem}[C.L.T. for the density of particles]\label{flu1}

\quad

Consider the Markov process $\{\eta_{t}: t\geq{0}\}$ starting from the invariant state $\nu_p$ under the assumption $g(n)=1+o(1)$.
Then, the sequence of processes $\{\mathcal{Y}_{t}^n\}_{ n\in{\bb N}}$ converges in distribution, as $n\rightarrow{+\infty}$, with respect to the
Skorohod topology
of $\mathcal{D}([0,T],\mathcal{S}'(\bb R))$ to   $\mathcal{Y}_t$
in $\mathcal{C}([0,T],\mathcal{S}'(\bb R))$, the generalized Ornstein-Uhlenbeck process of characteristics $\Delta,\nabla$ which is the formal solution of the equation \eqref{eq Ou}. 
  
  On the other hand, if we consider $g(n)=\alpha n^{-\beta}$, $\alpha>0$ and $\beta>1$,  then  
  $\{\mathcal{Y}_{t}^n\}_{ n\in{\bb N}}$ converges in distribution, as $n\rightarrow{+\infty}$, with respect to the
Skorohod topology
of $\mathcal{D}([0,T],\mathcal{S}_{Neu}'(\bb R))$ to   $\mathcal{Y}_t$
in $\mathcal{C}([0,T],\mathcal{S}_{Neu}'(\bb R))$, the generalized Ornstein-Uhlenbeck process of characteristics $\Delta_{\textrm{\rm Neu}},\nabla_{\textrm{\rm Neu}}$ which is the formal solution of the equation \eqref{eq Ou Neu}.
\end{theorem}

\section{Hydrodynamics}\label{s3}
  We proceed to define the spatial density of particles of the exclusion process, where we embed the discrete torus $\bb T_n$ in the continuous torus $\bb T$. \medskip

Let $\mc M$ be the space of positive measures on $\bb T$ with total mass bounded by one, endowed with the weak topology. Let
$\pi^{n}_{t} \in \mc M$ be the measure on $\bb T$ obtained by rescaling time by $n^2$, rescaling space by $n^{-1}$, and assigning mass $n^{-1}$ to each particle, i.e.,
\begin{equation}\label{f01}
\pi^{n}_{t}(\eta,du) \;=\; \pfrac{1}{n} \sum _{x\in \bb T_n} \eta_{t} (x)\,
\delta_{x/n}(du)\,,
\end{equation}
where $\delta_u$ is the Dirac measure concentrated on $u$. The usual name for $\pi^n_t(\eta,du)$ is \emph{empirical measure}.
For an integrable function
$H:\bb T \to \bb R$, the expression $\<\pi^n_t, H\>$ stands for
the integral of $H$ with respect to $\pi^n_t$:
\begin{equation*}
\<\pi^n_t, H\> \;=\; \pfrac 1n \sum_{x\in\bb T_n}
H (\pfrac{x}{n})\, \eta_{t}(x)\,.
\end{equation*}
This notation is not to be mistaken with the inner product in $L^2(\bb T)$. Also, when $\pi_t$ has a density
$\rho$, namely when $\pi(t,du) = \rho(t,u) du$, we sometimes write $\<\rho_t, H\>$
for $\<\pi_t, H\>$. \medskip

To avoid unwanted topological issues, in the entire paper a time horizon $T>0$ is fixed.  Let $\mc D([0,T], \mc M)$ be the space of $\mc M$-valued
c\`adl\`ag trajectories $\pi:[0,T]\to\mc M$ endowed with the
\emph{Skorohod} topology.  For each probability measure $\mu_n$ on
$\{0,1\}^{\bb T_n}$, denote by $\bb Q_{\mu_n}^{n}$ the measure on
the path space $\mc D([0,T], \mc M)$ induced by the measure $\mu_n$ and
the process $\pi^n_t$ introduced in \eqref{f01}.

Recall the profile $\gamma : \bb T \to [0,1]$ and the
sequence $\{\mu_n\}_{n\in \bb N}$ of measures on $\{0,1\}^{\bb T_n}$
defined through \eqref{eq288}. Let $\bb Q$ be
the probability measure on the space $\mc D([0,T], \mc M)$ concentrated on the deterministic path $\pi(t,du) = \rho (t,u)du$, where
 \begin{itemize}

\item
if  $g(n)=1+o(1)$, the function $\rho$ is the unique weak solution of \eqref{hep};
\vspace{0.2cm}

\item  if $g(n)=\frac{\alpha}{n^\beta}$, the function $\rho$ is the unique weak solution of  \eqref{hen}.
\end{itemize}


\begin{proposition}\label{prop:4.1}
Considering the two possibilities above for the function $g$, the sequence of probability measures $\{\bb
Q_{\mu_n}^{n}\}_{n\in\bb N}$ converges weakly to $\bb Q$  as $n\to\infty$.
\end{proposition}

Since Theorem \ref{th:1} is  an immediate corollary of the previous proposition, our goal  is to prove Proposition \ref{prop:4.1}.

The proof  is divided in several parts. In Subsection \ref{sub4.1},
we show that the sequence $\{\bb Q_{\mu_n}^{n}\}_{n\in \bb N}$ is tight. In Sections \ref{sub3.4}, \ref{sub3.5},  we show that, for each case of $g$, $\bb Q$  is the only possible limit along subsequences of  $\{\bb Q_{\mu_n}^{n}\}_{n\in \bb N}$. This  assures that the sequence
$\{\bb Q_{\mu_n}^{n}\}_{n\in\bb N}$ converges weakly to $\bb Q$, as $n\to\infty$.

\subsection{Tightness} \label{sub4.1}

In order to prove
tightness of $\{\pi^{n}_t : 0\le t \le T\}_{n\in \bb N}$ it is enough to
show tightness of the real-valued processes $\{\<\pi^{n}_t ,H\> :
0\le t \le T\}_{n\in \bb N}$ for a set of functions $H\in C(\bb{T})$, provided this set of functions is  dense in $C(\bb T)$ with respect to the uniform topology (see \cite[page 54, Proposition 1.7]{kl}). For that purpose, let $H\in C^2(\bb T)$. By Dynkin's formula,
\begin{equation}\label{M}
M^{n}_{t}(H)\;:=\;\<\pi^{n}_{t}, H\>- \<\pi^{n}_{0}, H\>-\int_{0}^{t}n^2\mf L_{n}\<\pi^{n}_{s},H\>\,ds
\end{equation}
is a martingale with respect to the natural filtration $\mathcal{F}_t:=\sigma(\eta_s: s\leq{t})$. Moreover, 
\begin{equation}\label{quadratic}
\big(M^{n}_{t}(H)\big)^2-\int_{0}^{t}\Big(n^2\mf L_{n} [\<\pi^{n}_{s},H\>]^2- 2\<\pi^{n}_{s},H\>\,n^2\mf
L_{n}\<\pi^{n}_{s},H\>\Big)\,ds
\end{equation}
is also a martingale with respect to the same filtration, see \cite{kl}.  In order to prove tightness of $\{\pi^{n}_t(H) : 0\le t \le T\}_{n\in\bb N}$, we shall prove
tightness
of each term in the formula above and then we invoke the fact that a sequence of a finite sum of tight processes is again tight.

Since \eqref{quadratic} is a martingale, doing  elementary calculations we obtain  the quadratic variation of $M^{n}_{t}(H)$ at time $T$ as
\begin{equation}\label{quad2}
\begin{split}
&\<M^{n}(H)\>_T= \int_{0}^{T} \sum_{x\in\bb T_{n}\backslash\{0\}} \Big((\eta_{s}(x)-\eta_{s}(x+1))
(H(\pfrac{x+1}{n})-H(\pfrac{x}{n}))\Big)^2ds\\
	& +\int_{0}^{T} \Big(\eta_{s}(1)(1-\eta_{s}(0))+g(n)\eta_{s}(0)(1-\eta_{s}(1))\Big)
(H(\pfrac{1}{n})-H(\pfrac{0}{n}))^2ds\\
	& +\int_{0}^{T} \Big(\eta_{s}(-1)(1-\eta_{s}(0))+g(n)\eta_{s}(0)(1-\eta_{s}(-1))
\Big)(H(\pfrac{-1}{n})-H(\pfrac{0}{n}))^2ds\,.\\
\end{split}
\end{equation}
The smoothness of $H$ implies that $\lim_{n\to\infty}\bb E_{\mu_n}\big[\<M^{n}(H)\>_T\big]=0$. Hence $M^{n}_{T}(H)$ converges
to zero  in $L^2(\mathbb{P}_{\mu_n})$ as  $n\to\infty$ and,  by Doob's inequality, for every $\delta>0$,
\begin{equation}\label{limprob}
\lim_{n\rightarrow\infty}\bb P_{\mu_n}\left\{\,\eta_\bola\,:\,\sup_{0\leq t\leq T} |M^{n}_{t}(H)|>\delta\right\}\;=\;0\,.
\end{equation}
In particular, this yields tightness of the sequence of martingales $\{M^{n}_{t}(H): 0\le t \le T\}_{n\in\bb N}$.

 A long computation, albeit completely elementary, shows us that the term $n^{2}\mf L_{n}\<\pi^{n}_{s},H\>$ appearing inside the time integral in (\ref{M}) can be rewritten as
\begin{equation}\label{integralterm}
\begin{split}
&n\!\!\!\!\!\!\sum_{x\in\bb T_n\backslash\{0\}}\!\!\!\!\!\!\!\!\Big(H(\pfrac{x+1}{n})+H(\pfrac{x-1}{n})-2H(\pfrac{x}{n})\Big)\eta_{s}(x)\\
+&ng(n)\Big(H(\pfrac{1}{n})+H(\pfrac{-1}{n})-2H(\pfrac{0}{n})\Big)\eta_{s}(0)\\
+&n(1-g(n))\Big[(H(\pfrac{-1}{n})-H(\pfrac{0}{n}))\,\eta_{s}(0)\,\eta_{s}(-1)+
(H(\pfrac{1}{n})-H(\pfrac{0}{n}))\,\eta_{s}(0)\,\eta_{s}(1)\Big]\,.\\
\end{split}
\end{equation}
We note that the first term above  corresponds to the discrete Laplacian leading to the heat
equation, while the other two terms arise from the boundary conditions.

By the smoothness of $H$ again, there exists a constant $c_H>0$ such that
$|n^2\mf L_{n}\<\pi^{n}_{s},H\>|\leq c_H$, which in turn gives
\begin{equation*}
 \left|\int_{r}^{t}n^2\mf L_{n}\<\pi^{n}_{s},H\>ds\right|\;\leq\; c_H\,|t-r|\,.
\end{equation*}
By the Arzel\`a-Ascoli Theorem the sequence of these integral terms is a relatively compact set, with respect to the  uniform topology, therefore it is tight. The term  $\<\pi^{n}_{0}, H\>$ is constant in time and bounded, thus is tight as well. This concludes the proof that the set of measures $\{\bb Q_{\mu_n}^{n}\}_{n\in \bb N}$ is tight.

\subsection{Entropy}\label{subinv}
Denote by ${ \bf H} (\mu | \nu_p)$ the entropy of a probability
measure $\mu$ with respect to the invariant state $\nu_p$. For a precise definition and properties of the entropy, we refer the reader to \cite{kl}.

\begin{proposition}\label{prop:4.2}
There exists a finite constant $K_0:=K_0(p)$, such that
\begin{equation*}
{\bf H} (\mu | \nu_p) \;\le\; K_0\, n\,,
\end{equation*}
for any probability measure $\mu$ on ${\{0,1\}^{\mathbb{T}_n}}$.
\end{proposition}
\begin{proof}
Recall the definition of $\nu_p$ and notice that
\begin{equation*}
\begin{split}
\mathbf{H} (\mu | \nu_p)&\;=\; \sum_{\eta\in\{0,1\}^{\bb T_n}}\mu(\eta)\log\Big(\frac{\mu(\eta)}{\nu_p(\eta)}\Big)
\;\leq\; \sum_{\eta\in\{0,1\}^{\bb T_n}}\mu(\eta)\log\Big(\frac{1}{\nu_p(\eta)}\Big)\,.\\
\end{split}
\end{equation*}
Recall \eqref{nub}. By the assumption $0<p<1$ and the inequality
\begin{equation*}
\nu_p(\eta)\geq (p\wedge (1-p))^{n-1}\,( m_p(0)\wedge(1-m_p(0)))\,
\end{equation*}
 we conclude that ${\bf H} (\mu | \nu_p)\leq K_0\,n$ for  some $K_0>0$ depending only on $p$.
\end{proof}

A remark: In  particular, the estimate ${\bf H} (\mu_n | \nu_p) \leq K_0\,n$ holds for the measures $\mu_n$  defined in \eqref{eq288}.

\subsection{Dirichlet form}\label{sub33}
Let $f$ be any density with respect to the invariant measure $\nu_p$. In others words,
 $f$ is a non negative function $f:\{0,1\}^{\bb T_n}\to \bb R$ satisfying $\int f(\eta)\nu_p(d\eta)=1$.
The Dirichlet form $\mf D_n$ is the convex and lower semicontinuous functional defined through
\begin{equation*}
\mf D_n(\sqrt{f}) =-\int \sqrt{f(\eta)}\,\mf L_n \sqrt{f(\eta)}\, \nu_p(d\eta)\,.
\end{equation*}
Invoking a general result \cite[Appendix 1, Prop. 10.1]{kl} we can write $\mf D_n$ as
\begin{equation}\label{Dirichlet}
\begin{split}
\mf D_n (\sqrt{f})\;=\;
&\frac{g(n)}{2}\int\eta(0)(1-\eta(1))\Big\{\sqrt{f(\eta^{0,1})}-\sqrt{f(\eta)}\Big\}^2\nu_p(d\eta)\\
+&\frac{1}{2}\int\eta(1)(1-\eta(0))\Big\{\sqrt{f(\eta^{0,1})}-\sqrt{f(\eta)}\Big\}^2\nu_p(d\eta)\\
+&\frac{1}{2}\int\eta(-1)(1-\eta(0))\Big\{\sqrt{f(\eta^{0,-1})}-\sqrt{f(\eta)}\Big\}^2\nu_p(d\eta)\\
+&\frac{g(n)}{2}\int\eta(0)(1-\eta(-1))\Big\{\sqrt{f(\eta^{0,-1})}-\sqrt{f(\eta)}\Big\}^2\nu_p(d\eta)\\
 +& \frac{1}{2}\sum_{\at{x\in \bb T_n}{x\neq 0,-1}}
\int\Big\{\sqrt{f(\eta^{x,x+1})}-\sqrt{f(\eta)}\Big\}^2\nu_p(d\eta)\,,\\
\end{split}
\end{equation}
where $\eta^{x,x+1}$ has been defined in \eqref{eta}.
\begin{proposition}\label{Prop3.3}
 Let $\bb Q^*$ be a limit of a subsequence of the sequence of probabilities measures $\{\bb Q^{n}_{\mu_n}\}_{n\in \bb N}$. Then $\bb
Q^*$ is concentrated  on trajectories $\pi(t,du)$ with a density with respect to the Lebesgue measure, i.e., of the
form $\pi(t,du) = \rho(t,u) du$. Moreover, the density $\rho(t,u)$ belongs to the space $L^2(0,T;\mc H^1)$, see Definition \ref{Sobolevdefinition}.
\end{proposition}
 The proof of the proposition above can be adapted from \cite[Proposition 5.6]{fgn1} and for this reason it will be omitted. For the interested reader, we briefly indicate some steps of this adaptation.

  We begin by observing that \cite[Proposition 5.6]{fgn1} is in fact a consequence of \cite[Lemma 5.8]{fgn1}. Thus we just describe how to prove, in our case, the statement in \cite[Lemma 5.8]{fgn1}.

 There are two basic ingredients in the proof of \cite[Lemma 5.8]{fgn1}.  The first one is that the entropy (with respect to the invariant measure) of any probability measure on the state space of the process, namely, $\{0,1\}^{\bb T_n}$   does not grow more than linearly. In our case, this result is proved in Proposition \ref{prop:4.2}.

The second ingredient in the proof of \cite[Lemma 5.8]{fgn1} is the fact  that, except at the defect,  the  Dirichlet form of the considered process coincides with the Dirichlet form of the homogeneous exclusion process. This fact indeed holds  for  the exclusion process with a  slow site  and therefore the same proof of  \cite[Lemma 5.8]{fgn1} applies here.

\subsection{Hydrodynamic limit for  $g(n)=1+o(1)$}\label{sub3.4}

\begin{proof}[Proof of Proposition \ref{prop:4.1}  for   $g(n)=1+o(1)$.]
Let $\bb Q^*$ be the weak limit
of some convergent subsequence  $\{\bb Q^{n_j}_{\mu_{n_j}}\}_{j\in \bb N}$ of the sequence  $\{\bb Q^{n}_{\mu_n}\}_{n\in \bb N}$. In order not to overburden the notation,  denote this subsequence just by $\{\bb Q^{n}_{\mu_{n}}\}_{n\in \bb N}$.
By Proposition \ref{Prop3.3},  the probability measure $\bb Q^*$ is concentrated on
trajectories
$\pi(t,du) =\rho(t,u) du$ such that $\rho(t,u)\in L^2(0,T;\mc H^1)$.
Our goal is to conclude that $\rho$  is a weak solution of the partial differential
equation \eqref{hep}.

Let $H\in C^2(\bb T)$.
We  claim  that
\begin{equation}\label{Q1}
 \bb Q^{*}\Big\{\,\pi:\,
  \<\pi_t, H \> \,-\,
 \<\pi_0,H \> \,-\,
 \int_0^t  \,  \<\pi_s , \partial_u^2 H \> \,ds \,=\,0,\,\forall t\in[0,T]\, \Big\}\;=\;1\,.
  \end{equation}
 To prove this, it suffices to show that
\begin{equation*}
\bb Q^{*} \Big\{\,\pi: \, \sup_{0\le t\le T}
\Big\vert \<\pi_t, H \> \,-\,
\<\pi_0,H \> \,-\,
\int_0^t  \,  \<\pi_s , \partial_u^2 H \> \,ds\, \Big\vert
 \, > \, \delta\, \Big\}\,=0\,,
 \end{equation*}
 for every $\delta >0$. Since the supremum is a continuous function in the Skorohod metric, by Portmanteau's Theorem, the probability above is smaller or equal  than
 \begin{equation*}
 \liminf_{n\to\infty} \bb Q^{n}_{\mu_n} \Big\{\,\pi: \, \sup_{0\le t\le T}
\Big\vert \<\pi_t, H \> \,-\,
\<\pi_0,H \> \,-\,
\int_0^t  \,  \<\pi_s ,\partial_u^2 H \> \,ds \,\Big\vert
 \, > \, \delta\,\Big\}\,.
\end{equation*}
Since $\bb Q^{n}_{\mu_n}$  is the measure on the space $\mc D([0,T], \mc M)$ induced by $\bb P_{\mu_n}^{n}$ via the empirical measure, we can rewrite the expression above as
\begin{equation*}
 \liminf_{n\to\infty} \bb P_{\mu_n} \Big\{\,\eta_\bola\,:\, \sup_{0\le t\le T}
\Big\vert \<\pi_t^n, H \> \,-\,
\<\pi_0^n,H \> \,-\,
\int_0^t  \,  \<\pi_s^n ,\partial_u^2 H \> \,ds \,\Big\vert
 \, > \, \delta\,\Big\}\,.
\end{equation*}

Adding and subtracting $n^2\,\mf L_n\<\pi^n_s, H\>$ to the integral term above, we can see that the previous expression is bounded from above by the sum of
\begin{equation}\label{eq413}
\begin{split}
&\limsup_{n\to\infty} \bb P_{\mu_n} \Big\{ \, \sup_{0\le t\le T}
\Big\vert \<\pi_t^n, H \> \,-\,
\<\pi_0^n,H \> \,-\,
\int_0^t  \,  n^2\,\mf L_n\<\pi^n_s, H\> \,ds\, \Big\vert
 \, > \, \delta/2\, \Big\}\\
 \end{split}
 \end{equation}
 and
 \begin{equation}\label{eq414}
\begin{split}
 & \limsup_{n\to\infty} \bb P_{\mu_n} \Big\{\, \sup_{0\le t\le T} \Big\vert \int_0^t
 \big(n^2\,\mf L_n\<\pi^n_s, H\> -  \<\pi_s^n ,\partial_u^2 H \>\big)\,ds \,\Big\vert
 \, > \, \delta/2 \,\Big\}\,.
\end{split}
\end{equation}
As already verified in Subsection \ref{sub4.1}, the quadratic variation of the martingale  $M^{n}_{t}(H)$ given in \eqref{M} goes to zero, as $n\to\infty$. Therefore, by Doob's inequality,  expression \eqref{eq413} is null.

It remains to show that \eqref{eq414} also vanishes. Recall \eqref{integralterm} for $n^2\,\mf L_n\<\pi^n_s, H\>$. Let us examine its terms.

The first term in the sum \eqref{integralterm} is
\begin{equation*}
n\sum_{x\in\bb T_n\backslash\{0\}}\Big(H(\pfrac{x+1}{n})+H(\pfrac{x-1}{n})-2H(\pfrac{x}{n})\Big)\eta_{s}(x)\,,
\end{equation*}
from which one can subtract
\begin{equation*}
\<\pi_s^n ,\partial_u^2 H \> \;=\;\pfrac{1}{n}\sum_{x\in \bb T_n}  \p_u^2 H(\pfrac{x}{n})\,\eta_s(x),
\end{equation*}
this difference being bounded (in modulus) by $c_H/n$, again because $H\in C^2(\bb T)$, where  $c_H>0$ is a constant depending only on $H$.

The second term in  \eqref{integralterm} is
\begin{equation*}
n \Big(1+o(1)\Big)
\Big(H(\pfrac{1}{n})+H(\pfrac{-1}{n})-2H(\pfrac{0}{n})\Big)\eta_{s}(0),
\end{equation*}
which converges to zero as $n\to\infty$, because $H\in C^2(\bb T)$.

The last term in \eqref{integralterm} is
\begin{equation*}
n(1-g(n))\Big[(H(\pfrac{-1}{n})-H(\pfrac{0}{n}))\,\eta_{s}(0)\,\eta_{s}(-1)+
(H(\pfrac{1}{n})-H(\pfrac{0}{n}))\,\eta_{s}(0)\,\eta_{s}(1)\Big]\,,
\end{equation*}
which goes to zero, as $n$ goes to infinity, because $H$ is smooth and $g(n)=1+o(1)$.
By the facts above we conclude that \eqref{eq414} is zero, proving the claim.

Now, let $\{H_i\}_{i\geq{1}}$ be a countable dense set of functions in $C^2(\bb T)$, with respect to the norm $\|H\|_{\infty}+\|\partial_uH\|_\infty+\|\partial_u^2H\|_\infty$. Intersecting a countable number of sets of probability one,
 \eqref{Q1} can be extended for
all functions $H\in C^2(\bb T)$ simultaneously, proving that $\bb Q^*$ is concentrated on weak solutions of \eqref{hep}. Since there  exists only one weak solution of \eqref{hep}, it means that $\bb Q^*$ is equal to the
aforementioned probability measure $\bb Q$. Invoking tightness  proved in Subsection \ref{sub4.1}, we conclude that the entire sequence $\{\bb Q^{n}_{\mu_n}\}_{n\in \bb N}$ converges to $\bb Q$ as $n\to\infty$.
\end{proof}

 \subsection{Law of large numbers for the occupation at the origin}\label{LGN}
  Recall the definition of $\mu_n$ given in \eqref{eq288} .
\begin{proposition}\label{prop:3.2}
Consider $g(n)=\alpha n^{-\beta}$. Let $\gamma:\bb T\to \bb [0,1]$ be a continuous profile, except possibly at $x=0$  and satisfying \eqref{assumption}. Then, for all $t>0$ and   $\eps>0$,
\begin{equation}\label{eq61}
\lim_{n\to \infty} \bb P_{\mu_n}\Big\{\,\eta_\bola\,:\,  \frac{n}{t}\int_{0}^{t}(1-\eta_{s}(0))\,ds  \,>\,\eps\, \Big\}\;=\;
0\,.
\end{equation}
\end{proposition}
This statement says that the  site $x=0$ remains empty a fraction of time smaller than $1/n$.
A simple heuristics for this statement is the following. The time a particles takes to escape the slow
site is, at least, an exponential random variable of parameter $g(n)$. If a random variable has exponential distribution of
parameter $\lambda$, its expectation is $1/\lambda$. Hence,  the time average  that a trapped particle takes to escape from the slow site
is at least $n^\beta/\alpha$ (if its neighboring sites are occupied, the trapped particle can spend even more time there). As time goes by, the slow site will remain empty a fraction of time at most $g(n)=\alpha n^{-\beta}$. Since $\beta>1$, this would lead to  \eqref{eq61}.

Despite the simplicity of the heuristics above, it is not straightforward to transform it into a rigorous argument. In order to prove Proposition \ref{prop:3.2}, we will make use of attractiveness and  the knowledge on the invariant measures. 
 For that purpose, in $\{0,1\}^{\bb T_n}$ we introduce the natural order between configurations: $\eta\leq \zeta$ if and only if $\eta(x)\leq \zeta(x)$ for all $x\in\bb T_n$.
A function $f:\{0,1\}^{\bb T_n}\to\bb R$ is said to be monotone if  $f(\eta)\leq f(\zeta)$ whenever $\eta\leq \zeta$.
This partial order is naturally extended to the space of measures.

We write $\mu_1 \leq_{\textrm{st}}\mu_2$ if, and only if,
\begin{equation*}
 \int f\,d\mu_1\;\leq \; \int f\,d\mu_2
\end{equation*}
for all monotone functions $f$. In this case, we say  that $\mu_1$ is
\emph{stochastically dominated from above} by $\mu_2$. The next result is well known and can found in \cite{liggett} for instance.
\begin{theorem}\label{th:3.2}
 Let $\mu_1$ and $\mu_2$ be two probability measures on $\{0,1\}^{\bb T_n}$. The statements below are equivalent:
 \begin{enumerate}
  \item $\mu_1\leq_{\textrm{st}} \mu_2$;
  \item There exists a probability measure $\bar{\mu}$ on $\{0,1\}^{\bb T_n}\times \{0,1\}^{\bb T_n}$ such that
  its first and second marginals are $\mu_1$ and $\mu_2$, respectively, and $\bar{\mu}$ is ``concentrated above the diagonal", which means
  \begin{equation*}
   \bar{\mu}\;\big\{ (\eta,\zeta)\,:\, \eta\leq \zeta\big \}\;=\;1\,.
  \end{equation*}
 \end{enumerate}
\end{theorem}
Next, we  construct such a measure $\bar{\mu}$, with the aforementioned  property, by means of the so-called \emph{graphical construction}.

Fix $n\in\bb N$. For each site $x$ of $\bb T_n$,  we associate two Poisson point processes $\mc N_{x}^{n,-}$ and $\mc N_{x}^{n,+}$, all of them being independent. The parameters of those Poisson process agree with the  Figure \ref{fig:1}. In other words, the parameter of $\mc N_{x}^{n,-}$
and of $\mc N_{x}^{n,+}$ is one for all $x\in\bb T_n$, except for $x=0$, for which
the parameter of $\mc N_{0}^{n,-}$ and $\mc N_{0}^{n,+}$ is equal to $g(n)$.

Given an initial configuration of particles $\eta\in\{0,1\}^{\bb T_n}$ and the ``toss" of those Poisson processes, the dynamics
will be the following. At a time arrival of some Poisson process, let us say, a time arrival of $\mc N_{x}^{n,-}$, if there is a
particle at the site $x$, and there is no particle at the site  $x-1$, the particle at $x$ moves to $x-1$. The
analogous happens with respect to a Poisson process of type $\mc N_{x}^{n,+}$, in which the movement (if possible) is from $x$ to
$x+1$. This construction yields the same Markov process previously defined via the generator given in  \eqref{generator}.

Consider now two probability measures $\mu_1$ and $\mu_2$ in $\{0,1\}^{\bb T_n}$ such that $\mu_1\leq_{\textrm{st}} \mu_2$.
By  Theorem \ref{th:3.2}, there exists a measure $\bar{\mu}$ on $\{0,1\}^{\bb T_n}\times \{0,1\}^{\bb T_n}$ concentrated
above the diagonal, as it is stated there. Evolving a configuration $(\eta^1,\eta^2)$, chosen by $\bar{\mu}$, by the same set of
Poisson point processes described above, we are lead to $\eta^1_t\leq \eta^2_t$, for any future time $t>0$. A stochastic
process enjoying this property of preserving the partial order  is said to be \emph{attractive}. See Liggett's book \cite{liggett}
for more details on the subject. 

 Notice that the specific value of $g(n)$ does not play any role in the argument above. In  resume, we can say that  the defect does not destroy attractiveness.

Having established attractiveness of the exclusion process with a slow site we make the following observation.  Once the Theorem \ref{th:1} is true for initial measures $\mu_n$ \emph{conditioned to have a particle at the origin}, the statement will remain in force for initial measures $\mu_n$. This is explained as follows.

By attractiveness, we can construct both processes (the one starting from $\mu_n$ and the one starting from $\mu_n$ conditioned to have a particle at the origin) in such a way that these processes will differ at most at one site, for any later time $t$. Therefore, the empirical measures \eqref{f01} for each process will have the same limit in distribution.

  Without loss of generality, we assume henceforth that there is a particle at the origin at the initial time.

\begin{proof}[Proof of Proposition \ref{prop:3.2}]
Since we  have assumed $\gamma(x)\geq \zeta>0$, for all $x\in\bb T$, since there is a particle at the origin and since $\mu_n$ is a product measure, we can find   $p>0$ small enough such  that $\mu_n\geq_{\textrm{st}} \nu_p$, for any $n\in\bb N$.

Fix $\eps>0$. By  attractiveness,
\begin{equation*}
\bb P_{\mu_n}\Big\{\,\eta_\bola\,:\,\frac{n}{t}\int_{0}^{t}\eta_{s}(0)\,ds>\eps\Big\}\;\geq \; \bb
P_{{p}}\Big\{\eta_\bola\,:\,\frac{n}{t}
\int_{0}^{t}\eta_{s}(0)\,ds>\eps\, \Big\}\,,
\end{equation*}
 which in turn implies
\begin{equation*}
\bb P_{\mu_n}\Big\{\,\eta_\bola\,:\,\frac{n}{t}\int_{0}^{t}(1-\eta_{s}(0))\,ds>\eps\Big\}\;\leq \; \bb
P_{p}\Big\{\,\eta_\bola\,:\,
\frac{n}{t}\int_{0}^{t}(1-\eta_{s}(0))\,ds>\eps\, \Big\}\,.
\end{equation*}

By  Chebyshev's inequality and Fubini's Theorem,
\begin{equation*}
\begin{split}
\bb
P_{p}\Big\{\,\eta_\bola\,:\, \frac{n}{t}\int_{0}^{t}(1-\eta_{s}(0))\,ds>\eps \,\Big\}\;&\leq \;  \frac{n}{\eps}\,\bb E_{p}\Big[\,\frac{1}{t}\int_{0}^{t} (1-\eta_{s}(0))\,ds\,\Big]\\
	&=\; \frac{n}{\eps}\,\Big(1-\frac{1}{t}\int_{0}^{t} \bb E_{p}[\eta_{s}(0)]\,ds\Big)\,.\\
\end{split}
\end{equation*}
Since
\begin{equation*}
 \nu_p\{\eta\;;\;\eta(0)=1\}\;=\;\frac{\pfrac{p}{g(n)}}{(1-p)+\pfrac{p}{g(n)}}\,,
\end{equation*}
we obtain that
\begin{equation*}
\bb P_{\mu_{n}}\Big\{\,\eta_\bola\,:\,\frac{n}{t}\int_{0}^{t}(1-\eta_{s}(0))\,ds>\eps\,\Big\}\;\leq\;
\frac{n}{\eps}\,\Big(1-\frac{\pfrac{p}{g(n)}}{(1-p)+\pfrac{p}{g(n)}}\Big)\,,\\
\end{equation*}
finishing the proof since $g(n)=\alpha n^{-\beta}$, $\beta>1$.
\end{proof}

\subsection{Hydrodynamic limit  for $g(n)=\alpha n^{-\beta}$}\label{sub3.5}
Recall that we have denoted $\lfloor \eps n\rfloor$, the integer part of $\eps n$, simply by $\eps n$.
Define the right average at $1\in \bb T_n$, and the left average at $-1\in \bb T_n$  by
\begin{equation}\label{lat_av}
 \eta^{\eps n,\textrm{R}}(1)\;:=\;\pfrac{1}{\eps n}\sum_{y=1}^{\eps n}\eta(y)\quad \textrm{ and }\quad
 \eta^{\eps n,\textrm{L}}(-1)\;:=\;\pfrac{1}{\eps n}\sum_{y=n-\eps n}^{n-1}\eta(y)\,,
\end{equation}
respectively.
Notice that none of these sums involve the occupation at the slow site $0\in \bb T_n$.
\begin{proposition}\label{Replacement}
For any $t>0$,
\begin{equation*}
 \limsup_{\eps\downarrow 0}\;\limsup_{n\to\infty}\;
\bb E_{\mu_n}\Big[\,\Big|\int_0^t\Big(\eta_{s}(1)- \eta^{\eps n,\,\textrm{R}}_s(1)\Big)\, ds\Big|\,\Big]\;=\;0
\end{equation*}
and
\begin{equation*}
 \limsup_{\eps \downarrow 0}\;\limsup_{n\to\infty}\;
\bb E_{\mu_n}\Big[\,\Big|\int_0^t\Big(\eta_{s}(-1)- \eta^{\eps n,\,\textrm{L}}_s(-1)\Big)\, ds\Big|\,\Big]\;=\;0\,.
\end{equation*}
\end{proposition}
The last result says that we can replace the occupation at the neighboring sites of the slow site by their averages in closed boxes, provided that these boxes do not cross the slow site. This kind of argument appears often in the literature and can be found, for example, in \cite{Landim}.

\begin{proof}
We treat the case $x=+1$,  the case $x=-1$ being analogous. From Jensen's inequality and the definition of the entropy, for any $N>0$,
the expectation appearing in the statement of this proposition is bounded from above by
\begin{equation}\label{log000}
\frac{{\bf H}(\mu_n|\nu_p)}{N n}+\frac{1}{N n}\log \bb E_{p} \Big[
\exp\Big\{N\,n\,\Big|\int_0^t \{
\eta_s(1)-\eta^{\eps n,\,\textrm{R}}_s(1)\}\,ds\Big|\Big\}\Big]\,.
\end{equation}
By Proposition \ref{prop:4.2}, ${\bf H}(\mu_n|\nu_p)\leq K_0\,n$, hence the term on the left hand side of last expression is bounded from above by $K_0/N$. Now, we bound the remaining term. Since $e^{|x|}\leq e^x+e^{-x}$ and
\begin{equation}\label{log bounds}
\limsup_n \pfrac{1}{n}\log (a_n+b_n)\;=\; \max\Big\{
\limsup_n \pfrac{1}{n}\log (a_n)\,,\,\limsup_n \pfrac{1}{n}\log(b_n)\Big\}\,,
\end{equation}
we can remove the modulus inside the exponential. Moreover, by the Feynman-Kac formula\footnote{See, for example, Lemma A1.7.2 of \cite{kl}.}
 the term on the right hand side of \eqref{log000}  is less than or equal to
\begin{equation*}
t\sup_{f\textrm{~density}}\Big\{\int
\{\eta(1)-\eta^{\eps n,\,\textrm{R}}_s(1)\}f(\eta)\nu_p(d\eta) - n\,\mf D_n(\sqrt{f})\Big\}\,ds\,.
\end{equation*}
Notice that the expression above does not depend on $N$. We claim now that, for any density $f$,
\begin{equation*}
\int
\{\eta(1)-\eta^{\eps n,\,\textrm{R}}_s(1)\}f(\eta)\nu_p(d\eta) \;\leq\; 2\eps+ n\,\mf D_n(\sqrt{f})
\end{equation*}
Since $N$ is arbitrary large, once we prove this claim    the proof will be finished. By the definition in \eqref{lat_av},
  \begin{equation*}
 \int
\{\eta(1)-\eta^{\eps n,\,\textrm{R}}_s(1)\}f(\eta)\nu_p(d\eta)\;=\;
 \int \Big\{\pfrac{1}{\eps n}\sum_{y=1}^{\eps n}(\eta(1)-\eta(y))\Big\}f(\eta)\,\nu_p(d\eta)\,.
 \end{equation*}
Writing $\eta(x)-\eta(y)$ as a telescopic sum, the right hand side of above can be rewritten as
 \begin{equation*}
  \int \Big\{\pfrac{1}{\eps n}\sum_{y=1}^{\eps n}
\sum_{z=1}^{y-1}(\eta(z)-\eta(z+1))\Big\}f(\eta)\,\nu_p(d\eta)\,.
 \end{equation*}
Rewriting the expression above as twice the half and making the transformation $\eta\mapsto \eta^{z,z+1}$ (for
which the probability $\nu_p$ is invariant) it becomes
\begin{equation*}
 \pfrac{1}{2\eps n}\sum_{y=1}^{\eps n}
\sum_{z=1}^{y-1}\int \{ \eta(z)-\eta(z+1)\}(f(\eta)-f(\eta^{z,z+1}))\,\nu_p(d\eta)\,.
 \end{equation*}
By $(a-b)=(\sqrt{a}-\sqrt{b})(\sqrt{a}+\sqrt{b})$ and Cauchy-Schwarz's inequality, we bound the previous expression from above by
\begin{equation*}
\begin{split}
 &\pfrac{1}{2\eps n}\sum_{y=1}^{\eps n}
\sum_{z=1}^{y-1}A\int\{\eta(z)-\eta(z+1)\}^2\Big(\sqrt{f(\eta)}+\sqrt{f(\eta^{z,z+1})}\Big)^2\,\nu_p(d\eta)\\\
 +\,&\pfrac{1}{2\eps n}\sum_{y=1}^{\eps n}
\sum_{z=1}^{y-1}\pfrac{1}{A}\int \Big(\sqrt{f(\eta)}-\sqrt{f(\eta^{z,z+1})}\Big)^ 2\,\nu_p(d\eta)\,,
\end{split}
 \end{equation*}
 for any $A>0$. Since $f$ is a density and recalling  \eqref{Dirichlet}, the expression above is bounded by $A^{-1}\mf D_n(f)+2A\eps n$.  
 Choosing $A=1/n$ we achieve the claim, concluding the proof.

\end{proof}

\begin{proof}[Proof of Proposition \ref{prop:4.1}  for $g(n)=\alpha n^{-\beta}$, $\alpha>0$, $\beta>1$.]
Again, let $\bb Q^*$ be the weak limit
of some convergent subsequence  $\{\bb Q^{n_j}_{\mu_{n_j}}\}_{j\in \bb N}$ of the sequence  $\{\bb Q^{n}_{\mu_n}\}_{n\in \bb N}$ and to keep notation simple,  denote this subsequence  by $\{\bb Q^{n}_{\mu_{n}}\}_{n\in \bb N}$.
Recall that by Proposition \ref{Prop3.3},  the probability measure $\bb Q^*$ is concentrated on
trajectories $\pi(t,du) =\rho(t,u) du$ such that $\rho(t,u)\in L^2(0,T;\mc H^1)$.
By the notion of trace in Sobolev spaces,  the integrals
\begin{equation}\label{boundary}
\int_0^t\rho(s,0)\,ds\qquad \textrm{and} \qquad \int_0^t\rho(s,1)\,ds
\end{equation}
are well defined and are finite. See \cite{e} for the properties of Sobolev spaces.
More than that, since the Sobolev space in one dimension is composed of functions 
which are absolutely continuous, $\rho$ has indeed left and right limits at zero.

Our goal here  is to conclude that $\rho$  is a weak solution of the partial differential
equation \eqref{hen}.

For this purpose, let $H\in C^2[0,1]$. Notice that, if $H$ is seen as a function in the torus, $H$ is possibly discontinuous at zero. We impose that $H(0)=0$.
We  claim  that
\begin{equation}\label{Q2}
\begin{split}
 \bb Q^{*}\Big\{\,\pi:\; &
  \<\rho_t, H \> \,-\,
 \<\rho_0,H \> \,-\,
 \int_0^t  \,  \<\rho_s , \partial_u^2 H \> \,ds \\
 & -\int_0^t\big(\rho_s(0)\,\partial_u H(0)-\rho_s(1)\,\partial_u H(1)\big)\,ds
 \,=\,0,\,\forall\, t\in[0,T]\, \Big\}\;=\;1\,.
  \end{split}
  \end{equation}
 To prove this, it suffices to show that
 \begin{equation*}
\begin{split}
 \bb Q^{*}\Big\{\,\pi:\; & \sup_{0\le t\le T}
\Big\vert
  \<\rho_t, H \> \,-\,
 \<\rho_0,H \> \,-\,
 \int_0^t  \,  \<\rho_s , \partial_u^2 H \> \,ds \\
 & -\int_0^t\big(\rho_s(0)\,\partial_u H(0)-\rho_s(1)\,\partial_u H(1)\big)\,ds
 \,\Big\vert
 \, > \, \delta\, \Big\}\;=\;0\,,
  \end{split}
  \end{equation*}
    for any $\delta>0$. Since the integrals in \eqref{boundary} are not defined in the whole Skorohod space $\mc D([0,T], \mc M)$, we cannot apply Portmanteau's Theorem yet.

  For that purpose,  let $\iota_\eps(u)=\eps^ {-1}\textbf{1}_{(0,\eps]}(u)$. Adding and subtracting the convolution of $\rho(t,u)$ with $\iota_\eps$ at the boundaries, we bound the previous probability by the sum of
     \begin{equation}\label{eq74}
\begin{split}
 \bb Q^{*}&\Big\{\,\pi:\;  \sup_{0\le t\le T}
\Big\vert
  \<\rho_t, H \> \,-\,
 \<\rho_0,H \> \,-\,
 \int_0^t  \,  \<\rho_s , \partial_u^2 H \> \,ds \\
 & -\int_0^t\big(\rho_s*\iota_\eps)(0)\,\partial_u H(0)\,ds +\int_0^t(\rho_s*\iota_\eps)(1-\eps)\,\partial_u H(1)\big)\,ds
 \,\Big\vert
 \, > \, \delta/2\, \Big\}  \end{split}
  \end{equation}
  and
   \begin{equation*}
\begin{split}
 \bb Q^{*}\Big\{\,\pi:\; & \sup_{0\le t\le T}
\Big\vert
 \int_0^t\big(\rho_s*\iota_\eps)(0)\,\partial_u H(0)\,ds -\int_0^t(\rho_s*\iota_\eps)(1-\eps)\,\partial_u H(1)\big)\,ds  \\
 & -\int_0^t\big(\rho_s(0)\,\partial_u H(0)\,ds +\int_0^t\rho_s(1)\,\partial_u H(1)\big)\,ds
 \,\Big\vert
 \, > \, \delta/2\, \Big\}\,.
  \end{split}
  \end{equation*}
    Since $\rho$ has left and right side limits, taking $\eps$ small, the previous probability goes to zero, as $n\to\infty$. It remains to bound \eqref{eq74}. By Portmanteau's Theorem and since there is at most one particle per site,  \eqref{eq74} is bounded from above by
 \begin{equation}\label{eq75}
 \begin{split}
  &\limsup_{n\to\infty} \; \bb Q^{n}_{\mu_n} \Big\{\,\pi:\;  \sup_{0\le t\le T}
\Big\vert
  \<\pi_t^n, H \> \,-\,
 \<\pi_0^n,H \> \,-\,
 \int_0^t  \,  \< \pi^n_s , \partial_u^2 H \> \,ds \\
 &  -\int_0^t\< \pi_t^n, \eps^{-1} \textbf{1}_{(0,\eps]}\>\,\partial_u H(0)\,ds +\int_0^t\< \pi_t^n, \eps^{-1} \textbf{1}_{(1-\eps,1]}\>\,\partial_u H(1)\,ds
 \,\Big\vert
 \, > \, \delta/2\,\Big\}\,.
 \end{split}
\end{equation}
    Noticing the identities
\begin{equation*}
 \eta^{\eps n,\textrm{R}}_s(1)\;=\;\< \pi_s^n, \eps^{-1} \textbf{1}_{(0,\eps]}\> \qquad \textrm{ and }\qquad   \eta^{\eps n,\textrm{L}}_s(-1)\;=\;\< \pi_s^n, \eps^{-1} \textbf{1}_{(1-\eps,1]}\>\,,
\end{equation*}
we can rewrite \eqref{eq75} as
 \begin{equation*}
 \begin{split}
  &\limsup_{n\to\infty} \; \bb P_{\mu_n} \Big\{\,\eta:\;  \sup_{0\le t\le T}
\Big\vert
  \<\pi_t^n, H \> \,-\,
 \<\pi_0^n,H \> \,-\,
 \int_0^t  \,  \< \pi^n_s , \partial_u^2 H \> \,ds \\
 &  \quad\quad  \quad\quad-\int_0^t \eta^{\eps n,\textrm{R}}_s(1)\,\partial_u H(0)\,ds +\int_0^t \eta^{\eps n,\textrm{L}}_s(-1)\,\partial_u H(1)\,ds
 \,\Big\vert
 \, > \, \delta/2\,\Big\}\,.
 \end{split}
\end{equation*}
Recalling Proposition \ref{Replacement}, in order to prove that the limit above is equal to zero, it is enough to show that the limit below is  null:
 \begin{equation*}
 \begin{split}
  &\limsup_{n\to\infty} \; \bb P_{\mu_n} \Big\{\,\eta:\;  \sup_{0\le t\le T}
\Big\vert
  \<\pi_t^n, H \> \,-\,
 \<\pi_0^n,H \> \,-\,
 \int_0^t  \,  \< \pi^n_s , \partial_u^2 H \> \,ds \\
 &  -\int_0^t \eta_s(1)\,\partial_u H(0)\,ds +\int_0^t \eta_s(-1)\,\partial_u H(1)\,ds
 \,\Big\vert
 \, > \, \delta/2\,\Big\}\,.
 \end{split}
\end{equation*}
 Adding and subtracting $n^2\,\mf L_n\<\pi^n_s, H\>$,  the previous expression is bounded from above by the sum of
\begin{equation}\label{eq76a}
\begin{split}
&\limsup_{n\to\infty} \bb P_{\mu_n} \Big\{ \, \sup_{0\le t\le T}
\Big\vert \<\pi_t^n, H \> \,-\,
\<\pi_0^n,H \> \,-\,
\int_0^t  \,  n^2\,\mf L_n\<\pi^n_s, H\> \,ds\, \Big\vert
 \, > \, \delta/4\, \Big\}\\
 \end{split}
 \end{equation}
 and
 \begin{equation}\label{eq77}
\begin{split}
  \limsup_{n\to\infty} \bb P_{\mu_n} \Big\{\, &\sup_{0\le t\le T} \Big\vert \int_0^t
 n^2\,\mf L_n\<\pi^n_s, H\>\,ds -  \int_0^t\<\pi_s^n ,\partial_u^2 H \>\,ds \\
&
 -\int_0^t \eta_s(1)\,\partial_u H(0)\,ds +\int_0^t \eta_s(-1)\,\partial_u H(1)\,ds \,\Big\vert
 \, > \, \delta/4 \,\Big\}\,.
\end{split}
\end{equation}
As can be easily verified, by the imposed conditions on the test function $H$, the quadratic variation of the martingale  $M^{n}_{t}(H)$ given in \eqref{M} goes to zero, as $n\to\infty$. Therefore, by Doob's inequality, \eqref{eq76a} is null.

It remains to show that \eqref{eq77} also vanishes. Recall \eqref{integralterm} for $n^2\,\mf L_n\<\pi^n_s, H\>$. Let us examine its terms.
 The first term in the sum \eqref{integralterm} is
\begin{equation*}
n\sum_{x\in\bb T_n\backslash\{0\}}\Big(H(\pfrac{x+1}{n})+H(\pfrac{x-1}{n})-2H(\pfrac{x}{n})\Big)\eta_{s}(x)\,,
\end{equation*}
which we split into the sum of
\begin{equation}\label{eq78a}
n\sum_{x\in\bb T_n\backslash\{-1,0,1\}}\Big(H(\pfrac{x+1}{n})+H(\pfrac{x-1}{n})-2H(\pfrac{x}{n})\Big)\eta_{s}(x)
\end{equation}
and
\begin{equation}\label{eq79a}
n \Big(H(\pfrac{2}{n})-2H(\pfrac{1}{n})\Big)\eta_{s}(1)+n \Big(H(\pfrac{-1}{n})-2H(\pfrac{-2}{n})\Big)\eta_{s}(-1)\,.
\end{equation}
The difference between \eqref{eq78a} and
\begin{equation*}
\<\pi_s^n ,\partial_u^2 H \> \;=\;\pfrac{1}{n}\sum_{x\in \bb T_n}  \p_u^2 H(\pfrac{x}{n})\,\eta_s(x)
\end{equation*}
in bounded (in modulus) by $c_H/n$,  because $H\in C^2[0,1]$, where $c_H>0$ is a constant depending only on $H$.
The second term in the sum \eqref{integralterm} is
\begin{equation*}
 n^{1-\beta}\Big(H(\pfrac{1}{n})+H(\pfrac{-1}{n})-2H(\pfrac{0}{n})\Big)\eta_{s}(0)
\end{equation*}
which converges to zero as $n\to\infty$, because $\beta>1$.
The last term in \eqref{integralterm} is
\begin{equation*}
n(1-g(n))\Big[(H(\pfrac{-1}{n})-H(\pfrac{0}{n}))\,\eta_{s}(0)\,\eta_{s}(-1)+
(H(\pfrac{1}{n})-H(\pfrac{0}{n}))\,\eta_{s}(0)\,\eta_{s}(1)\Big],
\end{equation*}
which can be rewritten as the sum of
\begin{equation}\label{eq78}
n(1-g(n))\Big[(H(\pfrac{-1}{n})-H(\pfrac{0}{n}))\,\eta_{s}(-1)+
(H(\pfrac{1}{n})-H(\pfrac{0}{n}))\,\eta_{s}(1)\Big]
\end{equation}
and
\begin{equation*}
n(\eta_s(0)-1)(1-g(n))\Big[(H(\pfrac{-1}{n})-H(\pfrac{0}{n}))\,\eta_{s}(-1)+
(H(\pfrac{1}{n})-H(\pfrac{0}{n}))\,\eta_{s}(1)\Big]\,.
\end{equation*}
By Proposition \ref{prop:3.2}, the time integral of the last term in the previous expression converges to zero in probability, as $n\to\infty$.
Since $H(\pfrac{0}{n})=0$ and $\beta>1$, the expression \eqref{eq79a} plus the expression \eqref{eq78} is equal to
\begin{equation*}
n \Big(H(\pfrac{2}{n})-H(\pfrac{1}{n})\Big)\eta_{s}(1)+n \Big(H(\pfrac{-1}{n})-H(\pfrac{-2}{n})\Big)\eta_{s}(-1)\,,
\end{equation*}
plus an error of order $n^{1-\beta}$. The expression above is, asymptotically in $n$, the same as
\begin{equation*}
\eta_s(1)\,\partial_u H(0) - \eta_s(-1)\,\partial_u H(1)\,,
\end{equation*}
whose time integral  cancels with the remaining time integrals of \eqref{eq77} and therefore proves that \eqref{eq77} vanishes. This concludes the claim \eqref{Q2}.

Now, let $\{H_i\}_{i\geq{1}}$ be a countable dense set of functions on $C^2[0,1]$, with respect to the norm $\|H\|_{\infty}+\|\p_uH\|_{\infty}+\|\partial_u^2H\|_\infty$. Since \eqref{Q2} is true for each one of these functions $H_i$, we can extend  \eqref{Q2}  for
all functions $H\in C^2[0,1]$ simultaneously by intersecting a countable number of sets of probability one. This proves that $\bb Q^*$ is concentrated on weak solutions of \eqref{hen}. Since there  exists only one weak solution of \eqref{hen}, it means that $\bb Q^*$ is equal to the
aforementioned probability measure $\bb Q$. Invoking the tightness that we have proved in Subsection \ref{sub4.1}, we conclude that the entire sequence $\{\bb Q^{n}_{\mu_n}\}_{n\in \bb N}$ converges to $\bb Q$, as $n\to\infty$.
\end{proof}

\section{Equilibrium density fluctuations}\label{s4}
In this section we prove Theorem \ref{flu1}.
Recall that here we take the process evolving on $\bb Z$
and  recall also   \eqref{flut}. By  Dynkin's formula,
\begin{equation}\label{M2}
\mc M^{n}_{t}(H)\;:=\; \mc Y^n_{t}(H)-  \mc Y^n_{0}(H)-\int_{0}^{t}n^2\mf L_{n} \mc Y^n_{s}(H)\,ds\,,
\end{equation}
is a martingale with respect to the natural filtration $\mathcal{F}_t:=\sigma(\eta_s: s\leq{t})$. Besides that,
\begin{equation}\label{quadratic2}
\Big(\mc M^{n}_{t}(H)\Big)^2-\int_{0}^{t}\Big(n^2\mf L_{n} \big(\mc Y^n_{t}(H)\big)^2- 2\,\mc Y^n_{s}(H)\,n^2\mf
L_{n}\,\mc Y^n_{s}(H)\Big)\,ds\,,
\end{equation}
is also a martingale with respect to the same filtration. By 
\[
\mc I^n_t(H)\;=\;\int_0^t n^2\mf L_{n} \mc Y^n_{s}(H)\,ds
\]
we denote the integral part in \eqref{M2}. First we observe  that
the expression  $$\sum_{x\in\bb Z}\Big(H(\pfrac{x+1}{n})+H(\pfrac{x-1}{n})-2H(\pfrac{x}{n})\Big) p$$ is  well defined since $H$ decays fast and is  equal to zero. Analogously to  \eqref{integralterm}, some direct calculations then  yield
\begin{equation}\label{eq433}
\begin{split}
&n^2\mf L_{n} \mc Y^n_{s}(H)  \;=\; n^{3/2}\sum_{x\neq -1,0,1}\Big[H(\pfrac{x+1}{n})+H(\pfrac{x-1}{n})-2H(\pfrac{x}{n})\Big]\bar{\eta}_s(x)\\
& + n^{3/2}\Big[H(\pfrac{2}{n})+H(\pfrac{0}{n})-2H(\pfrac{1}{n})\Big]\bar{\eta}_s(1) + n^{3/2}\Big[H(\pfrac{-2}{n})+H(\pfrac{0}{n})-2H(\pfrac{-1}{n})\Big]\bar{\eta}_s(-1)\\
&+n^{3/2} (1-g(n))\Big[
(H(\pfrac{1}{n})-H(\pfrac{0}{n}))\eta_s(0)\eta_s(1)+(H(\pfrac{-1}{n})-H(\pfrac{0}{n}))\eta_s(0)\eta_s(-1) \Big]\\
& + n^{3/2} g(n) \Big[H(\pfrac{1}{n})+H(\pfrac{-1}{n})-2H(\pfrac{0}{n})\Big]\eta_s(0)+ \Theta(n,p,H),\\
\end{split}
\end{equation}
where $\bar{\eta}_s(x)=\eta(x)-m_p(x)$ is the centered occupation variable, as in \eqref{flut} and
\begin{equation*}
\Theta(n,p,H)\;=\;-n^{3/2}\Big[H(\pfrac{1}{n})+H(\pfrac{-1}{n})-2H(\pfrac{0}{n})\Big]p.
\end{equation*}
The next two propositions  are direct calculations very similar to  \cite{fgn2}, and for this reason their proofs are omitted.
\begin{proposition}\label{prop41}
Consider $g(n)=1+o(1)$ and $H\in \mc S(\bb R)$. In this case,
\begin{equation*}
\lim_{n\to \infty }\bb E_p\Big[\big(\mc M^{n}_{t}(H)\big)^2\Big]\;=\; 2\,t\,\chi(p)\Vert \nabla H\Vert_{L^2(\bb R)}^2.
\end{equation*}
and
\begin{equation*}
\limsup_{n\to\infty} \bb E_p\Big[\big(\mc I^n_t(H)\big)^2\Big]\;\leq\;80\, t\,\chi(p)\Vert \nabla H\Vert^2_{L^2(\bb R)}
\end{equation*}
\end{proposition}
\begin{proposition}\label{prop42}
Consider $g(n)=\alpha n^{-\beta}$, $\alpha>0$, $\beta>1$, and $H\in \mc S_{\textrm{\rm Neu}}(\bb R)$. In this case,
\begin{equation*}
\lim_{n\to \infty }\bb E_p\Big[\Big(\mc M^{n}_{t}(H)\Big)^2\Big]\;=\;2 \,t\,\chi(p) \Vert \nabla_{\textrm{\rm Neu}} H\Vert_{L^2(\bb R)}^2
\end{equation*}
and 
\begin{equation*}
\limsup_{n\to\infty} \bb E_p\Big[\big(\mc I^n_t(H)\big)^2\Big]\;\leq\;80 \,t\,\chi(p)\Vert \nabla_{\textrm{\rm Neu}} H\Vert^2_{L^2(\bb R)}.
\end{equation*}
\end{proposition}
 
 The next result is concerned with convergence at initial time, for either case of the function $g(n)$.

\begin{proposition} \label{convergence at time zero}
 $\{\mc Y^n_0\}_{n\in\bb N}$ converges in distribution to $\mc Y_0$, as $n\to\infty$,
where $\mc Y_0$
 is a Gaussian field with mean zero and covariance given by \eqref{eq:covar1}.
\end{proposition}
\textit{Mutatis mutandis}, the same proof of \cite[Prop. 3.2]{fgn2} applies and is suppressed here.

\subsection{Tightness} Here we prove tightness of the process $\{\mc Y_t^n; t \in [0,T]\}_{n \in \bb N}$ in both cases of $g(n)$.
 First we notice that by  Mitoma's
criterion \cite{Mit} and the fact that $\mc S(\bb R)$ and $\mc S_{\textrm{\rm Neu}}(\bb R)$ are Fr\'echet spaces, it is enough to prove tightness of the sequence of real-valued processes $\{\mc Y_t^n(H); t \in [0,T]\}_{n \in \bb N}$,
where $H\in{\mc {S}(\bb R)}$ if we consider $g(n)=1+o(1)$, and $H\in\mc S_{\textrm{Neu}}(\bb R)$ if we consider $g(n)=\alpha n^{-\beta}$, $\alpha>0$, $\beta>1$.

 In order to prove tightness of $\{\mc Y^{n}_t(H) : 0\le t \le T\}_{n\in\bb N}$, we shall prove
tightness of each term in the formula \eqref{M2}.

Fix  a test function $H$ belonging to the respective space for each case of $g(n)$. By \eqref{M2}, it is enough to prove tightness of the stochastic processes $\{\mc Y_0^n(H)\}_{n \in
\bb N}$, $\{ \mc I_t^n(H); t \in [0,T]\}_{n \in \bb N}$, and $\{\mc M_t^n(H); t \in [0,T]\}_{n \in \bb N}$. 

By Proposition \ref
{convergence at time zero} we have convergence at initial time, hence $\{\mc Y_0^n(H)\}_{n \in
\bb N}$ is obviously tight.

To show tightness of the remaining  real-valued processes we use the  Aldous criterion:
\begin{proposition}[Aldous' criterion]
 A sequence $\{x_t^n; t\in [0,T]\}_{n \in \bb N}$ of real-valued processes is tight with respect to the Skorohod topology of $\mc
D([0,T],\bb R)$ if:
\begin{itemize}
\item[(i)]
$\displaystyle\lim_{A\rightarrow{+\infty}}\;\limsup_{n\rightarrow{+\infty}}\;\mathbb{P}\Big(\sup_{0\leq{t}\leq{T}}|x_{t
}^n |>A\Big)\;=\;0\,,$

\item[(ii)] for any $\varepsilon >0\,,$
 $\displaystyle\lim_{\delta \to 0} \;\limsup_{n \to {+\infty}} \;\sup_{\lambda \leq \delta} \;\sup_{\tau \in \mc T_T}\;
\mathbb{P}(|
x_{\tau+\lambda}^n- x_{\tau}^n| >\varepsilon)\; =\;0\,,$
\end{itemize}
where $\mc T_T$ is the set of stopping times bounded by $T$.
\end{proposition}

For the martingale term, the claim (i) of Aldous' criterion is achieved by an application of Doob's inequality
together with Proposition \ref{prop41} or Proposition \ref{prop42} (depending on the chosen $g$).

 By Proposition \ref{prop41} or Proposition \ref{prop42}, the  claim (i) of Aldous' criterion can be easily checked for the integral term. It
remains to check (ii). Fix a stopping time $\tau \in \mc T_T$ and suppose that $g(n)=1+o(1)$. By Chebychev's inequality,
\begin{equation*}
\begin{split}
\mathbb{P}_p\big(\big| \mc M_{\tau+\lambda}^n(H) - \mc M_\tau^n(H)\big| >\varepsilon\big)&\leq \frac{1}{\varepsilon^2} \mathbb{E}_p\big[ \big( \mc M_{\tau+\lambda}^n(H) - \mc M_\tau^n(H)\big)^2\big].\\
	\end{split}
\end{equation*}
Thus, by Proposition \ref{prop41},
\begin{equation*}
\begin{split}
\limsup_{n\to\infty} \mathbb{P}_p\big(\big| \mc M_{\tau+\lambda}^n(H) - \mc M_\tau^n(H)\big| >\varepsilon\big) & \leq  \frac{1}{\varepsilon^2} 2\chi(p)\,\lambda \Vert \nabla H \Vert^2_{L^2(\bb R)}\\
&\leq \frac{1}{\varepsilon^2} 2\chi(p)\,\delta \Vert  \nabla H\Vert^2_{L^2(\bb R)},
\end{split}
\end{equation*}
which vanishes as $\delta\rightarrow{0}$. 
Similarly, 
 \begin{equation*}
\begin{split}
\mathbb{P}_p\big(\big|  \mc I_{\tau+\lambda}^n(H) -  \mc I_\tau^n(H)\big| >\varepsilon\big)
	&\leq \frac{1}{\varepsilon^2} \mathbb{E}_p\big[ \big( \mc I_{\tau+\lambda}^n(H) - \mc I_\tau^n(H)\big)^2\big].\\
	\end{split}
\end{equation*}
Again by Proposition \ref{prop41}, we obtain
\begin{equation*}
\limsup_{n\to\infty}\mathbb{P}_p\big(\big|  \mc I_{\tau+\lambda}^n(H) -  \mc I_\tau^n(H)\big| >\varepsilon\big) \leq \frac{80t}{\varepsilon^2}\delta\,\chi(p)\Vert \nabla H\Vert^2_{L^2(\bb R)},
\end{equation*}
which vanishes as $\delta\rightarrow{0}$. 

The proof in the case $g(n)=\alpha n^{-\beta}$ is analogous  (invoking is this case Proposition \ref{prop42}) and for this reason will be omitted.
This finishes the proof of tightness.

\subsection{Characterization of limit points for  $g(n)=1+o(1)$}
We shall prove that any limit of $\{\mc Y^n_t(H)\}_{n\in \bb N}$ is concentrated on solutions of the martingale problem described in Proposition \ref{pp1}, with $H\in \mc S(\bb R)$. Suppose that $\{\mc Y^n_t\}_{n\in \bb N}$ converges along a subsequence to $\mc Y_t$. In slight abuse of notation, we denote this convergent subsequence also by $\{\mc Y^n_t\}_{n\in \bb N}$.

In this case $H\in \mc S(\bb R)$ is smooth, hence  we have 
\begin{equation*}
\Big| n^2\Big[H(\pfrac{x+1}{n})+H(\pfrac{x-1}{n})-2H(\pfrac{x}{n})\Big] -\Delta H(x)\Big|\leq \pfrac{c_H}{n}.
\end{equation*}
A similar analysis to the one presented in \eqref{eq433} for the hydrodynamic limit implies that
\begin{equation*}
\mc M^{n}_{t}(H)\;:=\; \mc Y^n_{t}(H)-  \mc Y^n_{0}(H)-\int_{0}^{t} \mc Y^n_{s}(\Delta H)\,ds + e(n)\,,
\end{equation*}
where the error function $e(n)$ is bounded, in modulus, by $cn^{-1/2}$. Since we are supposing that $\{\mc Y^n_t\}_{n\in \bb N}$ converges, we conclude that $\{\mc M_t^n(H)\}_{n\in \bb N}$ converges.

By similar arguments to those presented in \cite{fgn2}, we know that the sequence of martingales $\{\mc M^n_t(H)\}_{n\in \bb N}$ is uniformly integrable. This implies that the limit of $\{\mc M^n_t(H)\}_{n\in \bb N}$, which we denote by $\mc M_t(H)$, is again a martingale. By Proposition \ref{prop41}, its quadratic variation is $2\chi(p) \, t\,\Vert \nabla H\Vert_{L^2(\bb R)}^2$. Now, Proposition \ref{pp1} finishes the characterization of limit points in this case.

\subsection{Characterization of limit points for $g(n)=\alpha n^{-\beta}$}

We shall prove in this case that any limit of $\{\mc Y^n_t(H)\}_{n\in \bb N}$ with $H\in \mc S_{\textrm{Neu}}(\bb R)$ is a solution of the martingale problem described in Proposition \ref{pp2}.

In this situation, there is no  analogous result of Proposition \ref{prop:3.2}. The key ingredient here will be the following  tricky lemma:
\begin{lemma}\label{Lemma62} Let $g(n)=\alpha n^{-\beta}$ and $x=\pm 1$. For some constant $C>0$ not depending on $n$ the estimates
\begin{equation*}
\bb E_p\Bigg[ \Big(\int_0^t\Big(n^{3/2}\eta_s(x)(1-\eta_s(0))-\alpha n^{\frac{3}{2}-\beta}\eta_s(0)(1-\eta_s(x))\Big)\,ds\Big)^2\Bigg]\leq Cn^{1-\beta}\,,
\end{equation*}
hold.
\end{lemma}
\begin{proof} We prove only the inequality for $x=1$, the case $x=-1$ is completely analogous.
By the  Kipnis-Varadhan inequality (see \cite[Proposition A1.6.1]{kl}), the expectation
\begin{equation*}
\bb E_p\Bigg[ \Big(\int_0^t\Big(n^{3/2}\eta_s(1)(1-\eta_s(0))-\alpha n^{\frac{3}{2}-\beta}\eta_s(0)(1-\eta_s(1))\Big)\,ds\Big)^2\Bigg]
\end{equation*}
is less or equal than
\begin{equation*}
\begin{split}
   t \sup_{f\in L^2(\nu_p)}\Big\{\int n^{3/2}  \eta(1)& (1-\eta(0))f(\eta)\nu_p(d\eta) \\
& -
\alpha n^{\frac{3}{2}-\beta}\int\eta(0)(1-\eta(1))f(\eta)\nu_p(d\eta) -n^2\mf D_n(f)\Big\}.
\end{split}
\end{equation*}
where $\mf D_n$ is the Dirichlet form\footnote{Notice that here the Dirichlet form is evaluated at $f$ instead of $\sqrt{f}$ as in Subsection \ref{sub33}.} given in \eqref{Dirichlet}. In the first integral inside the supremum above we perform the change of variables $\eta\mapsto \eta^{1,0}$. Thus, the expression above can be rewritten as
\begin{equation}\label{eq63}
\begin{split}
   t \sup_{f\in L^2(\nu_p)}\Big\{\int n^{3/2}  \eta(0)& (1-\eta(1))f(\eta^{0,1})\nu_p(d\eta^{0,1}) \\
& -
\alpha n^{\frac{3}{2}-\beta}\int\eta(0)(1-\eta(1))f(\eta)\nu_p(d\eta) -n^2\mf D_n(f)\Big\}.
\end{split}
\end{equation}
Now, from \eqref{nub} we have that
\begin{equation*}
\frac{\nu_p(\{\eta^{0,1}\})}{\nu_p(\{\eta\})} = \Big(\pfrac{m_p(0)}{p}\Big)^{\eta(1)}\Big(\pfrac{p}{m_p(0)}\Big)^{\eta(0)}\Big(\pfrac{1-p}{1-m_p(0)}\Big)^{1-\eta(0)}\Big(\pfrac{1-m_p(0)}{1-p}\Big)^{1-\eta(1)}
\end{equation*}
from where we get
 \begin{equation*}
 \int n^{3/2}  \eta(0) (1-\eta(1))f(\eta^{0,1})\nu_p(d\eta^{0,1})=
  \int  \alpha n^{\frac{3}{2}-\beta} \eta(0)(1-\eta(1))f(\eta^{1,0})\nu_p(d\eta)\,,
 \end{equation*}
and therefore \eqref{eq63} is the same as
\begin{equation*}
t \sup_{f\in L^2(\nu_p)}\Big\{\alpha n^{\frac{3}{2}-\beta}\int \eta(0)(1-\eta(1))\Big(f(\eta^{1,0})-f(\eta)\Big)\nu_p(d\eta) -n^2\mf D_n(f)\Big\}\,.
\end{equation*}
 By the  inequality $xy\leq \pfrac{Ax^2}{2}+\pfrac{y^2}{2A}$, $\forall A>0$, the expression above is smaller than
\begin{equation*}
\begin{split}
t \sup_{f\in L^2(\nu_p)}\Big\{\frac{A\alpha n^{\frac{3}{2}-\beta}}{2}\int \eta(0)(1-\eta(1))&\Big(f(\eta^{1,0})-f(\eta)\Big)^2\nu_p(d\eta)\\ &+\frac{\alpha n^{\frac{3}{2}-\beta}}{2A}\int \eta(0)(1-\eta(1)) \nu_p(d\eta)-n^2\mf D_n(f)\Big\}
\end{split}
\end{equation*}
for any $A>0$. Picking $A= \sqrt{n}$ the last expression becomes equal to 
\begin{equation*}
\begin{split}
t \sup_{f\in L^2(\nu_p)}\Big\{\frac{\alpha n^{2-\beta}}{2}\int \eta(0)(1-\eta(1))&\Big(f(\eta^{1,0})-f(\eta)\Big)^2\nu_p(d\eta)\\ &+\frac{\alpha  n^{{1-\beta}}}{2}\int \eta(0)(1-\eta(1)) \nu_p(d\eta)-n^2\mf D_n(f)\Big\}.
\end{split}
\end{equation*}
Since the Dirichlet form  \eqref{Dirichlet} is a sum of positive terms, and since the first term above is exactly the first term in $n^2\mf D_n(f)$, we conclude that the expression above is less or equal than
\begin{equation*}
\frac{t\alpha n^{1-\beta}}{2}\int \eta(0)(1-\eta(1)) \nu_p(d\eta) =  Cn^{1-\beta},
\end{equation*}
for $C>0$  not depending on $n$.
This concludes the proof.
\end{proof}

Let us proceed to the characterization of limit points in this case. We begin with an observation that will strongly simplify the analysis. 

First of all, we notice that $\mc M^n_t(H)$ defined in \eqref{M2} is a martingale where $H\in \mc S_{\textrm{Neu}}(\bb R)$ does not play any special role, except the decay at infinity to make the sum well defined. 
In fact, we can take $H=H_n$ depending on $n$. We will do that in the following way. For each $n\in \bb N$, we impose $H_n(\pfrac{x}{n})=H(\pfrac{x}{n})$ for all $x\neq 0$ while for $x=0$
 we impose
\begin{equation}\label{eq451}
H_n(\pfrac{0}{n})=\pfrac{1}{2}\Big(H(\pfrac{1}{n})+
H(\pfrac{-1}{n})\Big).
\end{equation}
In this way, we obtain 
\begin{equation}\label{eq455}
H_n(\pfrac{1}{n})+H_n(\pfrac{-1}{n})-2H_n(\pfrac{0}{n})=0\;,\;\forall n\in \bb N
\end{equation}
 which cancels two parcels in \eqref{eq433}. To simplify notation we will write $H$ instead of $H_n$, keeping in mind  \eqref{eq455}.

We examine carefully all the  terms in \eqref{eq433}. By the discussion above, both $\Theta(n,p,H)$ and
\begin{equation}\label{zero}
n^{3/2}g(n)\Big[ H(\pfrac{1}{n})+ H(\pfrac{-1}{n})-2 H(\pfrac{0}{n})\Big]\eta(0)
\end{equation}
vanish. Let us see the remaining terms.  The first sum on the right hand side of \eqref{eq433} is equal to 
\begin{equation*}
\frac{1}{\sqrt{n}}\sum_{x\neq -1,0,1}\Delta_{\textrm{Neu}}H(\pfrac{x}{n})\,\bar{\eta}_s(x)
\end{equation*}
plus an error of order $O(n^{-1/2})$. 
Since the side derivatives of $H\in\mc S_{\textrm{Neu}}(\bb R)$ at zero vanish, we also have that the second and third terms in \eqref{eq433} are equal to 
\begin{equation*}
\begin{split}
&n^{3/2}\Big[H(\pfrac{0}{n})-H(\pfrac{1}{n})\Big]\bar{\eta}_s(1) + n^{3/2}\Big[H(\pfrac{0}{n})-H(\pfrac{-1}{n})\Big]\bar{\eta}_s(-1)\\
\end{split}
\end{equation*}
plus another error of order $O(n^{-1/2})$. By \eqref{eq451}, the expression above
can be rewritten as
\begin{equation*}
\frac{n^{3/2}}{2}\Big[ H(\pfrac{-1}{n})-H(\pfrac{1}{n})\Big]\eta(1)+\frac{n^{3/2}}{2}\Big[ H(\pfrac{1}{n})-H(\pfrac{-1}{n})\Big]\eta(-1).
\end{equation*}
 Last expression  together with the remaining two parcels in \eqref{eq433} gives us the sum of 
\begin{equation}\label{eq47}
\frac{n^{3/2}}{2}\Big[ H(\pfrac{-1}{n})-H(\pfrac{1}{n})\Big]\eta(1)+\frac{n^{3/2}}{2} (1-g(n))\Big[
H(\pfrac{1}{n})-H(\pfrac{-1}{n})\Big]\eta_s(0)\eta_s(1)
\end{equation}
and
\begin{equation}\label{eq48}
\frac{n^{3/2}}{2}\Big[ H(\pfrac{1}{n})-H(\pfrac{-1}{n})\Big]\eta(-1)+\frac{n^{3/2}}{2} (1-g(n))\Big[
H(\pfrac{-1}{n})-H(\pfrac{1}{n})\Big]\eta_s(0)\eta_s(-1).
\end{equation}
At this point, we will use \eqref{zero}. Regardless of the fact 
that \eqref{zero} is null, we can split it in two parts, namely $n^{3/2}g(n)\Big[H(\pfrac{1}{n})-H(\pfrac{0}{n})\Big]\eta(0)$ and $n^{3/2}g(n)\Big[H(\pfrac{-1}{n})-H(\pfrac{0}{n})\Big]\eta(0)$. The first one we add to \eqref{eq47} and the second one to \eqref{eq48}. Recalling \eqref{eq451}, it gives us the sum of
\begin{equation}\label{eq410}
\frac{n^{3/2}}{2}\Big(H(\pfrac{-1}{n})-H(\pfrac{1}{n})\Big)
\Big[\eta_s(1)(1-\eta_s(0))-g(n)\eta_s(0)(1-\eta_s(1))\Big]
\end{equation}
and
\begin{equation}\label{eq411}
\frac{n^{3/2}}{2}\Big(H(\pfrac{1}{n})-H(\pfrac{-1}{n})\Big)
\Big[\eta_s(-1)(1-\eta_s(0))-g(n)\eta_s(0)(1-\eta_s(-1))\Big].
\end{equation}
Lemma \ref{Lemma62} asserts that the time integrals of expressions \eqref{eq410} and \eqref{eq411} are asymptotically negligible in $L^2$.

For $H\in \mc S_{\textrm{Neu}}(\bb R)$, the sequence of martingales $\{\mc M^n_t(H)\}_{n\in \bb N}$ presented in \eqref{M2} is uniformly integrable. This implies that the $L ^2$-limit  of $\{\mc M^n_t(H)\}_{n\in \bb N}$, denoted by $\mc M_t(H)$, is again a martingale which quadratic variation given $2\chi(p)\,t\, \Vert \nabla_{\textrm{Neu}} H\Vert_{L^2(\bb R)}^2$, assured by Proposition \ref{prop41}.

The entire previous discussion on the integral part of $\mc M^n_t(H)$ lead us to conclude that $\mc M_t(H)$ satisfies
\begin{equation*}
\mc M_{t}(H)\;:=\; \mc Y_{t}(H)-  \mc Y_{0}(H)-\int_{0}^{t} \mc Y_{s}(\Delta_{\textrm{Neu}} H)\,ds\,,
\end{equation*}
which concludes the characterization of limit points by Proposition \ref{pp2}.

\section{Open questions and conjectures}\label{s5}
As presented in this paper, the  critical defect strength 
and behaviour at the critical point
remains open in sense that it  is not clear what should be the limit for $0\leq \beta\leq 1$ when $g(n)=\alpha n^{-\beta}$, $\alpha>0$. One  conceivable scenario is that for any $\beta>0$,  both hydrodynamic limit and fluctuations would be driven by a disconnect behavior 
 corresponding to  Neumann boundary conditions, meaning that the critical point
would be $\beta_c=0$.

 Our guess instead is that the correct critical point should be achieved at $\beta=1$, much more close to the scenario of \cite{fgn1} where a {\it slow bond} is considered instead of a {\it slow site}. 
 The physical intuition behind the  dynamical phase transition taking place at $\beta=1$ is the fact
that, in a large but finite system and at large but finite times, the particle current, which is of
diffusive origin, will be of order $1/n$ everywhere (with some space-dependent amplitude that depends on the
initial state) before equilibrium is reached. However, a weak site with $\beta>1$
cannot allow such a current to flow and hence it acts like a total blockage corresponding to Neumann boundary conditions. On the other hand, a defect rate with $\beta<1$ does not
make a current of order $1/n$ impossible, corresponding to a macroscopically irrelevant local
perturbation of the particle system.

  Specifically we conjecture that the behavior for $\beta=1$ should be described by the partial differential equation 
  \begin{equation}\label{PDE2}
\begin{cases}
 \partial_t \rho(t,u) \; =\; \p_{u}^2 \rho(t,u)\,,&t \geq 0,\, u\in (0,1)\,,\\
 \partial_u \rho(t,0^+) \; =\;\partial_u \rho(t,0^-)= \frac{\alpha}{2}(\rho(t,0^+)-\rho(t,0^-))\,,&t \geq 0\,,\\
  \rho(0,u) \;=\; \rho_0(u)\,, &u \in (0,1)\,,\\
\end{cases}
\end{equation}
where $0^+$ and $0^-$ denotes right and left side limits, respectively. 
This conjecture 
is motivated by the observation that the equation above is the hydrodynamic equation of two neighbouring slow bonds at $\beta=1$, a claim made precise in next section. 
Thus our conjecture in the slow site setting is the following.
\begin{itemize}
\item   For $\alpha>0$ and $0\leq \beta<1$, the hydrodynamic limit and the equilibrium fluctuations for $g(n)=\alpha n^{-\beta}$ should be driven by the heat equation
 periodic boundary conditions, achieving the same limits we have obtained for 
$g(n)=1+o(n)$.
\item For $\alpha>0$ and $\beta=1$, the  hydrodynamic limit and the equilibrium fluctuations for $g(n)=\alpha n^{-\beta}$ should be driven by  \eqref{PDE2} in the same sense  of \cite{fgn1,fgn2,fgn3} with the correction of $1/2$ in the boundary condition.
\end{itemize}

\section{Hydrodynamics of the SSEP with $k$ neighboring slow bonds}\label{s6}

Here we characterize the hydrodynamic behavior of the SSEP with $k$ neighboring slow bonds. This is an additional result we append in order to support the conjecture presented in Section \ref{s5}.   We point out that, for the regime $\beta=1$, the result presented here is 
{\it not} a corollary of \cite{fgn1,fgn2}, since here we consider $k$ \textit{neighboring} slow bonds, while the mentioned references considered \textit{macroscopically separated} slow bonds. 

The notation and topology issues will be the same as those  we have considered in this paper. Fix $k$ a positive integer. The SSEP with $k$ neighboring slow bonds is the Markov process on $\{0,1\}^{\bb T_n}$ defined through the generator
\begin{equation*}
\begin{split}
\mf L_{n}f(\eta)\;=\;&\sum_{x=0}^{k-1}\,\pfrac{\alpha}{n^\beta}\,[f(\eta^{x,x+1})-f(\eta)]+\!\!\!\!\!\!   \sum_{\at{x\in \bb T_n}{x\notin \{0,\ldots,k-1\}}}[f(\eta^{x,x+1})-f(\eta)]
\end{split}
\end{equation*}
acting on functions $f:\{0,1\}^{\bb T_n}\rightarrow \bb{R}$.
\begin{figure}[H]
  \centering
  \includegraphics{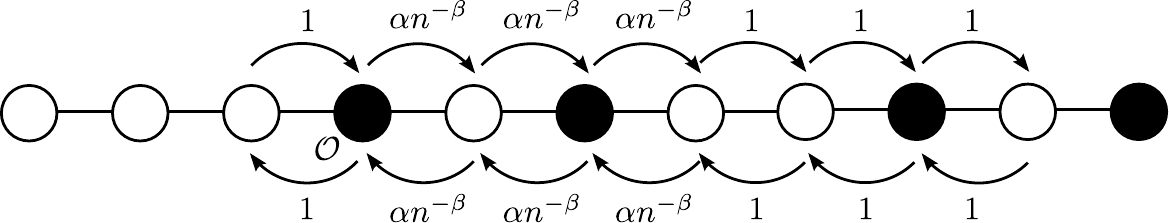}
  \label{fig:3}
\caption{Exclusion process with three neighboring {\emph{slow bonds}}.}
\end{figure}
\begin{definition}\label{heat equation Robin_k}
 Let $\rho_0:\bb T\to [0,1]$ be a measurable function. We say that $\rho$ is a weak solution of the  heat equation with Robin's boundary conditions given by
 \begin{equation}\label{her_k}
\begin{cases}
 \partial_t \rho(t,u) \; =\; \p_{u}^2 \rho(t,u)\,,&t \geq 0,\, u\in (0,1)\,,\\
 \partial_u \rho(t,0^+) \; =\;\partial_u \rho(t,0^-)= \frac{\alpha}{k}(\rho(t,0^+)-\rho(t,0^-))\,,&t \geq 0\,,\\
  \rho(0,u) \;=\; \gamma(u)\,, &u \in (0,1)\,,\\
\end{cases}
\end{equation}
if $\rho$ belongs to $L^2(0,T;\mathcal{H}^1)$ and for all $t\in [0,T]$ and for all  $H\in \A$, such that
\begin{equation}\label{eq82}
\p_u H(0^+)=\p_u H(0^-)=\pfrac{\alpha}{k}(H(0^+)-H(0^-))\,,
\end{equation}
holds that
\begin{equation*}
\begin{split}
&\< \rho_t\,,\,H\>-\<\gamma\,,\,H\> - \int_0^t\big\< \rho_s\,,\,\p^2_u H\big\>\, ds
=\;0\,.\\
\end{split}
\end{equation*}
\end{definition}

\begin{proposition}\label{th:k}
For each $n\in\bb N$, let $\mu_n$ be a Bernoulli product measure on $\{0,1\}^ {\bb T_n}$ as in \eqref{eq288}.
   Then, for any $t>0$, for every $\delta>0$ and every $H\in C(\bb{T})$, it holds that
\begin{equation}\label{eq:2.2ap}
\lim_{n\to\infty}
\bb P_{\mu_n} \Big\{\eta_\bola : \, \Big\vert \pfrac{1}{n} \sum_{x\in\bb{T}_n}
H(\pfrac{x}{n})\, \eta_{t}(x) - \int_{\bb T} H(u)\, \rho(t,u) du \Big\vert
> \delta \Big\} \;=\; 0\,,
\end{equation}
where
 \begin{itemize}
\item
if $0\leq \beta<1$, the function $\rho$ is the unique weak solution of \eqref{hep};
\vspace{0.1cm}

\item
 if $\beta= 1$, the function $\rho$ is the unique weak solution of  \eqref{her_k};
 \vspace{0.1cm}

\item  if $\beta>1$, the function $\rho$ is the unique weak solution of  \eqref{hen}.
\end{itemize}
\end{proposition}
  \begin{proof}
  As usual, the proof consists in proving tightness of the process induced by the empirical measure, plus the uniqueness of the limit points.

  Tightness can be handled in same way as we have done in Subsection \ref{sub4.1}. Characterization of the limit points for the cases $\beta\in [0,1)$ and $\beta\in{(1,\infty)}$ follows closely the steps of \cite{fgn1,fgn2}. The case $\beta=1$ is tricky  and
  described in more detail  below.

  Let $G_n:\{\pfrac{0}{n},\pfrac{1}{n},\ldots,\pfrac{n-1}{n}\}\to\bb R$ be some function depending on $n$.
  Performing elementary computations,
  $n^{2}\mf L_{n}\<\pi^{n}_{s},G_n\>$  can be rewritten as the sum of
\begin{equation}\label{eq83}
\begin{split}
&\Big[\,n\Big(G_n(\pfrac{k+1}{n})-G_n(\pfrac{k}{n})\Big)+
\alpha n^{1-\beta}\Big(G_n(\pfrac{k-1}{n})-G_n(\pfrac{k}{n})\Big)\Big]\,\eta_s(k)\\
+&\,\sum_{j=1}^{k-1}\Big[\,\alpha n^{1-\beta}\Big(G_n(\pfrac{j+1}{n})-G_n(\pfrac{j}{n})\Big)+
\alpha n^{1-\beta}\Big(G_n(\pfrac{j-1}{n})-G_n(\pfrac{j}{n})\Big)\Big]\,\eta_s(j)\\
+&\,\Big[\,\alpha n^{1-\beta}\Big(G_n(\pfrac{1}{n})-G_n(\pfrac{0}{n})\Big)+
 n\Big(G_n(\pfrac{-1}{n})-G_n(\pfrac{0}{n})\Big)\Big]\,\eta_s(0)\\
\end{split}
\end{equation}
and
\begin{equation}\label{eq84a}
\,n\!\!\!\!\!\!\sum_{\at{x\in \bb T_n}{x\notin \{0,\ldots,k\}}}\!\!\!\!\!\!\Big(G_n(\pfrac{x+1}{n})+G_n(\pfrac{x-1}{n})-2G_n(\pfrac{x}{n})\Big)\,\eta_{s}(x)\,.
\end{equation}
Let   $H\in \A$ satisfying \eqref{eq82}. We define $G_n$ by
\begin{equation}\label{eq86}
G_n(\pfrac{x}{n})\;=\;
\begin{cases}
\, H(\pfrac{x}{n})\, ,&  \textrm{ if }\quad  x\in \{k+1,\ldots, n-1\}\,,\\
H(0^-)+ \frac{x}{k}(H(0^+)-H(0^-)) \, ,&  \textrm{ if }\quad  x\in \{0,\ldots, k\}\,.\\
\end{cases}
\end{equation}
In other words, the function $G_n$ is equal to $H$ outside the region where the slow bonds are contained. At the sites $\{0,1,\ldots,k\}$, the function $G_n$ is a linear interpolation of $H(0^+)$ and $H(0^-)$.

Since $H\in C^2[0,1]$, \eqref{eq84a} is close to $\< \pi^n_s, \p^2_u H\>$.
We claim now that  \eqref{eq83} converges to zero, as $n\to\infty$. First, notice that \eqref{eq86} tells us that
\begin{equation*}
\sum_{j=1}^{k-1}\Big[\,\alpha n^{1-\beta}\Big(G_n(\pfrac{j+1}{n})-G_n(\pfrac{j}{n})\Big)+
\alpha n^{1-\beta}\Big(G_n(\pfrac{j-1}{n})-G_n(\pfrac{j}{n})\Big)\Big]\,\eta_s(j)
\end{equation*}
is null. Let us analyze the remaining terms in \eqref{eq83}. Since $\beta=1$, the term
\begin{equation*}
n\Big(G_n(\pfrac{k+1}{n})-G_n(\pfrac{k}{n})\Big)+
\alpha n^{1-\beta}\Big(G_n(\pfrac{k-1}{n})-G_n(\pfrac{k}{n})\Big)
\end{equation*}
converges to
\begin{equation*}
\p_u H(0^+)+
\pfrac{\alpha}{k} \Big(H(0^-)-H(0^+)\Big)\,,
\end{equation*}
which vanishes by \eqref{eq82}. The same analysis assures  that 
\begin{equation*}
\alpha n^{1-\beta}\Big(G_n(\pfrac{1}{n})-G_n(\pfrac{0}{n})\Big)+
 n\Big(G_n(\pfrac{-1}{n})-G_n(\pfrac{0}{n})\Big)
\end{equation*}
converges to zero as $n\to\infty$.
  Provided by this claim and similar arguments of those in Section \ref{s5} one can conclude the proof.
  \end{proof}

\section*{Acknowledgements}
TF was supported through a project PRODOC/UFBA and a project JCB 1708/2013-FAPESB.

PG thanks CNPq (Brazil) for support through the research project
``Additive functionals of particle systems'', Universal n. 480431/2013-2, also thanks FAPERJ ``Jovem Cientista do Nosso Estado''
for the grant E-25/203.407/2014 and the Research Centre of Mathematics of the University of Minho, for the financial support provided by ``FEDER" through the ``Programa Operacional Factores de Competitividade  COMPETE" and by FCT through the research project PEst-C/MAT/UI0013/ 2011.

GMS acknowledges financial support by DFG.

 The authors thank CMAT at the University of Minho, where this work was initiated, and 
PUC-Rio, where this work was finished, for the warm hospitality.

\bibliographystyle{amsplain}

\end{document}